\documentclass{amsart}
\usepackage[hmargin=30mm, vmargin=30mm]{geometry}

\usepackage[english]{babel}
\usepackage{amsmath,amssymb}
\usepackage{lmodern}
\numberwithin{equation}{section}
\usepackage{mathtools}
\usepackage{hyperref}

\newtheorem{thmintro}{Theorem}

\newtheorem{corintro}[thmintro]{Corollary}
\newtheorem{theorem}{Theorem}[section]

\newtheorem{lemma}[theorem]{Lemma}
\newtheorem{prop}[theorem]{Proposition}
\theoremstyle{definition}
\newtheorem{definition}[theorem]{Definition}
\newtheorem{remark}[theorem]{Remark}
\newtheorem{convention}[theorem]{Convention}
\newtheorem*{notation}{Notation}

\newcommand{\CC}{\mathbb{C}}
\newcommand{\NN}{\mathbb{N}}
\newcommand{\QQ}{\mathbb{Q}}
\newcommand{\RR}{\mathbb{R}}
\newcommand{\ZZ}{\mathbb{Z}}
\newcommand{\LLL}{\mathcal{L}}
\newcommand{\TTT}{\mathcal{T}}
\newcommand{\WW}{\mathcal{W}}

\newcommand{\g}{\mathfrak{g}}
\newcommand{\kk}{\mathfrak{k}}
\newcommand{\mm}{\mathbf{m}}

\newcommand{\HH}{\mathcal{H}}
\newcommand{\PP}{\mathcal{P}}
\newcommand{\uu}{\mathfrak{u}}
\newcommand{\ttt}{\mathfrak{t}}
\newcommand{\bc}{\mathbf{c}}
\newcommand{\bd}{\mathbf{d}}

\newcommand{\inv}{^{-1}}
\newcommand{\la}{\lambda}
\newcommand{\co}{\colon\thinspace}

\DeclareMathOperator{\sppan}{span}

\DeclareMathOperator{\fin}{\mathrm{fin}}
\DeclareMathOperator{\GL}{GL}
\DeclareMathOperator{\End}{End}
\DeclareMathOperator{\Aut}{Aut}
\DeclareMathOperator{\id}{id}
\DeclareMathOperator{\der}{der}
\DeclareMathOperator{\su}{sum}

\DeclarePairedDelimiter\ceil{\lceil}{\rceil}
\DeclarePairedDelimiter\floor{\lfloor}{\rfloor}
\DeclarePairedDelimiter\fr{\langle}{\rangle}
\DeclarePairedDelimiter\bfr{\big\langle}{\big\rangle}

\begin{document}

\renewcommand{\proofname}{{\bf Proof}}

\title[Positive energy representations of Hilbert loop algebras]{Positive energy representations of double extensions \\ of Hilbert loop algebras}
\author[Timoth\'ee Marquis]{Timoth\'ee \textsc{Marquis}$^*$}
\address{Department Mathematik, FAU Erlangen-Nuernberg, Cauerstrasse 11, 91058 Erlangen, Germany}
\email{marquis@math.fau.de}
\thanks{$^*$Supported by a Marie Curie Intra-European Fellowship}

\author[Karl-Hermann Neeb]{Karl-Hermann \textsc{Neeb}$^\dagger$}
\address{Department Mathematik, FAU Erlangen-Nuernberg, Cauerstrasse 11, 91058 Erlangen, Germany}
\email{neeb@math.fau.de}
\thanks{$^\dagger$Supported by DFG-grant NE 413/7-2, Schwerpunktprogramm ``Darstellungstheorie"}

\begin{abstract}
A real Lie algebra with a compatible Hilbert space structure (in the sense that the scalar product is invariant) is called a Hilbert--Lie algebra.  
Such Lie algebras are natural infinite-dimensional analogues of the compact Lie algebras; in particular, any infinite-dimensional simple Hilbert--Lie algebra $\mathfrak{k}$ is of one of the four classical types $A_J$, $B_J$, $C_J$ or $D_J$ for some infinite set~$J$. Imitating the construction of affine Kac--Moody algebras, one can then consider affinisations of $\mathfrak{k}$, that is, double extensions of (twisted) loop algebras over $\mathfrak{k}$. Such an affinisation $\mathfrak{g}$ of $\mathfrak{k}$ possesses a root space decomposition with respect to some Cartan subalgebra $\mathfrak{h}$, whose corresponding root system yields one of the seven locally affine root systems (LARS) of type $A_J^{(1)}$, $B^{(1)}_J$, $C^{(1)}_J$, $D_J^{(1)}$, $B_J^{(2)}$, $C_J^{(2)}$ or $BC_J^{(2)}$.

Let $D\in\mathrm{der}(\mathfrak{g})$ with $\mathfrak{h}\subseteq\mathrm{ker}D$ (a diagonal derivation of $\mathfrak{g}$). Then every highest weight representation $(\rho_{\lambda},L(\lambda))$ of $\mathfrak{g}$ with highest weight $\lambda$ can be extended to a representation $\widetilde{\rho}_{\lambda}$ of the semi-direct product $\mathfrak{g}\rtimes \RR D$. 
In this paper, we characterise all pairs $(\lambda,D)$ for which the representation $\widetilde{\rho}_{\lambda}$ is of positive energy, namely, for which the spectrum of the operator $-i\widetilde{\rho}_{\lambda}(D)$ is bounded from below.
\end{abstract}

\maketitle

\section{Introduction}
Let $G$ be a Lie group with Lie algebra $\g$, and let $\alpha\co\RR\to \Aut(G): t\mapsto\alpha_t$ define a continuous $\RR$-action on $G$. 
Consider a unitary representation $\pi\co G^{\sharp}\to U(\HH)$ of the topological group $G^{\sharp}:=G\rtimes_{\alpha}\RR$ on some Hilbert space $\HH$, and let 
$\mathrm{d}\pi\co\g\rtimes \RR D\to \uu(\HH)$ denote the corresponding derived representation, where $D:=\tfrac{d}{dt}|_{t=0}\mathrm{L}(\alpha_t)\in\der(\g)$ is the infinitesimal generator of $\alpha$. The representation $(\pi,\HH)$ is said to be of \emph{positive energy} if the spectrum of the \emph{Hamiltonian} $H:=-i\mathrm{d}\pi(D)$ is bounded from below.

It is a challenging natural problem to determine the irreducible positive energy representations $(\pi,\HH)$ of $G^{\sharp}$. 
As a consequence of the Borchers--Arveson Theorem (\cite[Theorem~3.2.46]{BR02}), 
for any such representation, 
the restriction $\rho:=\pi|_{G}$ is irreducible (see \cite[Theorem~2.5]{Ne14}), and the Hamiltonian $H$ of the extension of $\rho$ to $G^{\sharp}$ is determined by $\alpha$, up to an additive constant.
The set of irreducible positive energy representations of $G^{\sharp}$ may thus be viewed as a subset $\widehat{G}_{\alpha}$ of the set $\widehat{G}$ of equivalence classes of irreducible unitary representations of $G$, and one would like to describe this subset as explicitely as possible.

A prominent class of unitary representations which will be studied in this paper is provided by the subset $\widehat{G}_{hw}$ of irreducible unitary ``highest weight representations'' of $G$. 
In \cite{PEClocfin}, an explicit description of the positive energy representations in $\widehat{G}_{hw}$ was obtained for some prototypical example of a \emph{Hilbert--Lie group} $G$ (that is, such that the Lie algebra of $G$ is a Hilbert--Lie algebra), when $\alpha$ is given by conjugation with diagonal operators.
In this paper, we push this study further by considering double extensions of \emph{Hilbert loop groups}, that is, double extensions of loop groups over a Hilbert--Lie group.
Since the positive energy condition is expressed in terms of the derived representation $\mathrm{d}\pi\co \g\rtimes \RR D\to \uu(\HH)$, we will formulate our results at the level of the corresponding Lie algebras, namely, for double extensions of \emph{Hilbert loop algebras}. For the construction of highest weight representations of double extensions of Hilbert loop groups, we refer to \cite{Hloopgroups}.

\smallskip

We now present in more detail the main result of this paper. For a more thorough account of the concepts presented below, we refer to \cite[Sections~2 and 3]{PECisom} and the references therein.

A \emph{Hilbert--Lie algebra} is a real Lie algebra $\kk$ admitting a real Hilbert space structure with invariant scalar product, that is, such that $\langle [x,y],z\rangle=\langle x,[y,z]\rangle$ for all $x,y,z\in \kk$. Hilbert--Lie algebras generalise the classical matrix algebras to infinite rank. In particular, any simple infinite-dimensional Hilbert--Lie algebra $\kk$ possesses a root space decomposition with respect to some maximal abelian subalgebra $\ttt$ (a \emph{Cartan subalgebra}), whose corresponding root system $\Delta=\Delta(\kk,\ttt)\subseteq i\ttt^*$ is a so-called \emph{locally finite root system}, of one of the types $A_J$, $B_J$, $C_J$ or $D_J$ for some infinite set $J$ (see \cite{NeSt01} and \cite{LN04}). 

Let $\varphi\in\Aut(\kk)$ be an automorphism of the simple Hilbert--Lie algebra $\kk$ of finite order $N$, and let $\ttt_0$ be a maximal abelian subalgebra of $\kk^{\varphi}:=\{x\in\kk \ | \ \varphi(x)=x\}$. Then $\kk$ also possesses a root space decomposition $\kk_{\CC}=(\ttt_0)_{\CC}\oplus\bigoplus_{\alpha\in\Delta_{\varphi}}{\kk_{\CC}^{\alpha}}$ with respect to $\ttt_0$, with corresponding root system $\Delta_{\varphi}=\Delta(\kk,\ttt_0)$.

The (\emph{$\varphi$-twisted}) \emph{Hilbert loop algebra} over $\kk$ is the Lie algebra
$$\LLL_{\varphi}(\kk):=\{\xi\in C^{\infty}(\RR,\kk) \ | \ \xi(t+\tfrac{2\pi}{N})=\varphi\inv(\xi(t)) \ \forall t\in\RR\}.$$ 
We equip its complexification $\LLL_{\varphi}(\kk)_{\CC}$ with the invariant positive definite hermitian form $\langle\cdot,\cdot\rangle$ defined by
$$\langle \xi,\eta\rangle=\frac{1}{2\pi}\int_{0}^{2\pi}{\langle \xi(t),\eta(t)\rangle \mathrm{dt}}.$$

Let $\der_0(\LLL_{\varphi}(\kk),\langle\cdot,\cdot\rangle)$ denote the space of skew-symmetric derivations $D$ of $\LLL_{\varphi}(\kk)$ that are \emph{diagonal}, in the sense that $D(e^{int}\otimes \kk_{\CC}^{\alpha})\subseteq e^{int}\otimes \kk_{\CC}^{\alpha}$ for all $n\in\ZZ$ and $\alpha\in\Delta_{\varphi}$. Define $D_0\in \der_0(\LLL_{\varphi}(\kk),\langle\cdot,\cdot\rangle)$ by $D_0(\xi)=\xi'$. For any weight $\nu\in i\ttt_0^*$, let also $\overline{D}_{\nu}$ be the derivation of $\kk_{\CC}$ defined by 
$$\overline{D}_{\nu}(x_{\alpha}):=i\nu(\alpha^{\sharp})x_{\alpha}\quad\textrm{for all $x_{\alpha}\in\kk_{\CC}^{\alpha}$, $\alpha\in\Delta_{\varphi}$,}$$
where $\alpha^{\sharp}$ is the unique element of $i\ttt_0$ such that $\langle h,\alpha^{\sharp}\rangle=\alpha(h)$ for all $h\in \ttt_0$. Then $\overline{D}_{\nu}$ restricts to a skew-symmetric derivation of $\kk$, which we extend to a derivation in $\der_0(\LLL_{\varphi}(\kk),\langle\cdot,\cdot\rangle)$ by setting $\overline{D}_{\nu}(\xi)(t):=\overline{D}_{\nu}(\xi(t))$ for all $\xi\in \LLL_{\varphi}(\kk)$ and $t\in\RR$. The space $\der_0(\LLL_{\varphi}(\kk),\langle\cdot,\cdot\rangle)$ is then spanned by $D_0$ and all such $\overline{D}_{\nu}$ (see \cite[Theorem~7.2 and Lemma~8.6]{MY15}), and we set
$$D_{\nu}:=D_0+\overline{D}_{\nu}\in \der_0(\LLL_{\varphi}(\kk),\langle\cdot,\cdot\rangle).$$
The derivation $D_{\nu}$ defines a $2$-cocycle $\omega_{D_{\nu}}(x,y):=\langle D_{\nu}(x),y\rangle$ on $\LLL_{\varphi}(\kk)$, and extends to a derivation $\widetilde{D}_{\nu}(z,x):=(0,D_{\nu}(x))$ of the corresponding central extension $\RR\oplus_{\omega_{D_{\nu}}}\LLL_{\varphi}(\kk)$. We call the resulting double extension
$$\g=\widehat{\LLL}_{\varphi}^{\nu}(\kk):=(\RR\oplus_{\omega_{D_{\nu}}}\LLL_{\varphi}(\kk))\rtimes_{\widetilde{D}_{\nu}}\RR$$
the \emph{$\nu$-slanted and $\varphi$-twisted affinisation} of the Hilbert--Lie algebra $\kk$. 
The Lie algebra $\g$ admits a root space decomposition $\g_{\CC}=(\ttt_0^e)_{\CC}\oplus\bigoplus_{\alpha\in \widehat{\Delta}_{\varphi}}\g_{\alpha}$ with respect to its maximal abelian subalgebra $\ttt_0^e:=\RR\oplus\ttt_0\oplus\RR$, with corresponding root system $$\widehat{\Delta}_{\varphi}=\Delta(\g,\ttt_0^e)\subseteq \{0\}\times i\ttt_0^* \times \ZZ\subseteq i(\ttt_0^e)^*.$$ The set $$(\widehat{\Delta}_{\varphi})_c:=\{(0,\alpha,n)\in \widehat{\Delta}_{\varphi} \ | \ \alpha\neq 0\}\subseteq \{0\}\times \Delta_{\varphi} \times \ZZ$$ of \emph{compact roots} is then a so-called \emph{locally affine root system} (LARS). These LARS were classified in \cite{YY08}, and those of infinite rank fall into $7$ distinct families of isomorphism classes, parametrised by the types $X_J^{(1)}$ and $Y_J^{(2)}$ for $X\in\{A,B,C,D\}$ and $Y\in\{B,C,BC\}$, for some infinite set $J$. The type $X_J^{(1)}$ can be realised as the root system of the unslanted and untwisted affinisation $\widehat{\LLL}^{0}_{\id}(\kk)$ of some Hilbert--Lie algebra $\kk$ with root system of type $X_J$. The type $Y_J^{(2)}$ can similarly be realised as the root system of some unslanted and $\psi_Y$-twisted affinisation of a suitable Hilbert--Lie algebra $\kk$, for some automorphism $\psi_Y$ of order $2$ whose description can be found in \cite[\S 2.2]{Hloopgroups} (see also \cite[Section~6]{PECisom}). We call the three automorphisms $\psi_Y$, as well as the $7$ affinisations of a Hilbert--Lie algebra described above \emph{standard}.

Set $\bc:=(i,0,0)\in i\ttt_0^e\subseteq \g_{\CC}$ and $\bd:=(0,0,-i)\in i\ttt_0^e\subseteq\g_{\CC}$.
Let $\la\in i(\ttt_0^e)^*$, which we write as $\la=(\la_c,\la^{0},\la_d)$ where
$$\la_c:=\la(\bc)\in\RR, \quad \la^{0}:=\la|_{i\ttt_0}\in i\ttt_0^* \quad\textrm{and}\quad \la_d:=\la(\bd)\in\RR.$$ 
Assume that $\la_c\neq 0$ and that $\la$ is \emph{integral}, in the sense that $\la$ only takes integer values on the coroots $\check{\alpha}$, $\alpha\in \widehat{\Delta}_{\varphi}$ (see \S\ref{section:preliminaries:affine} below). Then $\g$ admits an irreducible integrable highest weight representation
$$\rho_{\la}\co \g\to\End(L(\la))$$
with highest weight $\la$, whose set of weights is given by $\PP_{\la}=\mathrm{Conv}(\widehat{\WW}_{\varphi}^{\nu}.\la)\cap (\la+ \ZZ[\widehat{\Delta}_{\varphi}]),$ where $\widehat{\WW}_{\varphi}^{\nu}=\WW(\g,\ttt_0^e)$ denotes the Weyl group of $\g$ with respect to $\ttt_0^e$ (see \cite[Theorem~4.10]{Ne09}). In fact, \cite[Theorem~4.11]{Ne09} even implies that $\rho_{\la}$ is unitary with respect to some (uniquely defined up to a positive factor) inner product on $L(\la)$.

Let $\nu'\in i\ttt_0^*$, and extend the derivation $D_{\nu'}$ of $\LLL_{\varphi}(\kk)$ to a skew-symmetric derivation of $\g$ by $D_{\nu'}(\ttt_0^e):=\{0\}$. Then $D_{\nu'}$ is encoded by the character $$\chi=\chi_{\nu'}\co\ZZ[\widehat{\Delta}_{\varphi}]\to\RR:(0,\alpha,n)\mapsto n+\nu'(\alpha^{\sharp})$$
satisfying $D_{\nu'}(x_{\alpha})=i\chi(\alpha)x_{\alpha}$ for all $x_{\alpha}\in\g_{\alpha}$, $\alpha\in\widehat{\Delta}_{\varphi}$.
One can now extend $\rho_{\la}$ to a representation  
$$\widetilde{\rho}_{\la,\chi}\co \g\rtimes \RR D_{\nu'}\to\End(L(\la))$$
of the semi-direct product $\g\rtimes\RR D_{\nu'}$, where $\widetilde{\rho}_{\la,\chi}(D_{\nu'})v_{\gamma}=i\chi(\gamma-\la)v_{\gamma}$ for all $\gamma\in \PP_{\la}$ and $v_{\gamma}\in L(\la)$ of weight $\gamma$. The representation $\widetilde{\rho}_{\la,\chi}$ is thus of \emph{positive energy} if and only if
$$M_{\g,\nu'}:=\inf \mathrm{Spec}(H_{\nu'})=\inf\chi\big(\PP_{\la}-\la\big)=\inf \chi\big(\widehat{\WW}_{\varphi}^{\nu}.\la-\la\big)>-\infty,$$
where $H_{\nu'}:=-i\widetilde{\rho}_{\la,\chi}(D_{\nu'})$ is the corresponding Hamiltonian.

We first characterise the positive energy highest weight representations of $\g$ when $\g$ is standard. In this case, there is an orthonormal basis $\{e_j \ | \ j\in J\}$ of $i\ttt_0$ such that the linearly independent system $\{\epsilon_j \ | \ j\in J\}\subseteq i\ttt_0^*$ defined by $\langle \epsilon_j,e_k\rangle=\delta_{jk}$ contains the root system $\Delta_{\varphi}$ in its $\ZZ$-span (see \S\ref{section:preliminaries:affine} below). Write a character $\chi\co i(\ttt_0^e)^*\to\RR$ as $\chi=(\chi_c,\chi^0,\chi_d)$ where
$$\chi_c:=\chi((1,0,0))\in\RR, \quad \chi^0:=\chi|_{i\ttt_0^*} \quad\textrm{and}\quad \chi_d:=\chi((0,0,1))\in\RR.$$
We call $\chi=(\chi_c,\chi^0,\chi_d)$ \emph{summable} if $\chi_c=\chi_d=0$ and $\chi^0\in\ell^1(J)$, that is, $\sum_{j\in J}{|\chi^0(\epsilon_j)|}<\infty$.

\begin{thmintro}\label{thmintro:A}
Let $(\g,\ttt^e_0)$ be a standard affinisation of a simple Hilbert--Lie algebra, with Weyl group $\widehat{\WW}=\WW(\g,\ttt^e_0)$.
Let $\la=(\la_c,\la^{0},\la_d)\in i(\ttt^e_0)^*$ be an integral weight with $\la_c\neq 0$. Then for any character $\chi=(\chi_c,\chi^0,\chi_d)\co i(\ttt_0^e)^*\to\RR$ with $\la_c\chi_d> 0$, the following are equivalent:
\begin{enumerate}
\item
$\inf \big(\chi\big(\widehat{\WW}.\la-\la\big)\big)>-\infty$.
\item
$\chi=\chi_{\min}+\chi_{\su}$ for some minimal energy character $\chi_{\min}$, satisfying $\inf \big(\chi_{\min}\big(\widehat{\WW}.\la-\la\big)\big)=0$, and some summable character $\chi_{\su}$.
\end{enumerate}
\end{thmintro}
In addition, we give an explicit description of the set of characters $\chi$ of minimal energy (see Section~\ref{Section:COTPFLZD}). An alternative description of this set is given in \cite[Theorem~3.5]{convexhull}. 
Note that the assumption $\la_c\chi_d> 0$ in Theorem~\ref{thmintro:A} is only necessary to avoid degenerate cases, which are dealt with in \S\ref{section:trivial_cases}.
The proof of Theorem~\ref{thmintro:A} relies on the earlier work \cite{PEClocfin}, which provides a similar characterisation of the positive energy condition for highest weight representations of Hilbert--Lie algebras.

To characterise the positive energy highest weight representations of $\g=\widehat{\LLL}_{\varphi}^{\nu}(\kk)$ arbitrary, we use the main results of \cite{PECisom}, which allows to reduce the problem to the ``standard'' case.
\begin{corintro}\label{corintro:arbitrary}
The statement of Theorem~\ref{thmintro:A} holds for arbitrary affinisations $\g=\widehat{\LLL}_{\varphi}^{\nu}(\kk)$ of a simple Hilbert--Lie algebra $\kk$.
\end{corintro}

A more precise statement of Corollary~\ref{corintro:arbitrary} is given in Theorem~\ref{thm:CorBprecise} below: the key point here is that the explicit form of the isomorphism from $\g$ to one of the (slanted) standard affinisations of $\kk$ provided by \cite[Theorem~A]{PECisom} also allows for an explicit description of the minimal energy sets in this more general setting.

\section{Preliminaries}

\begin{notation}
In this paper, we denote by $\NN=\{1,2,\dots\}$ the set of positive natural numbers.
\end{notation}

\subsection{Locally affine root systems}\label{section:preliminaries:affine}
The general reference for this paragraph is  \cite[\S 2.3 and \S 3.5]{PECisom} (and the references therein).

Let $J$ be an infinite set. Let $V_{\fin}:=\RR^{(J)}\subseteq V:=\RR^J$ be the free vector space over $J$, with canonical basis $\{e_j \ | \ j\in J\}$ and standard scalar product given by $\langle e_j,e_k\rangle=\delta_{jk}$. Note that we may extend $\langle\cdot,\cdot\rangle$ to a bilinear form on $V_{\fin}\times V$.  In the dual space $(V_{\fin})^*\cong \RR^{J}$ of $V_{\fin}$, we consider the linearly independent system $\{\epsilon_j \ | \ j\in J\}$ defined by $\epsilon_j(e_k)=\delta_{jk}$, and we denote by $(\cdot,\cdot)$ the standard scalar product on $$V^{\star}_{\fin}:=\sppan_{\RR}\{\epsilon_j \ | \ j\in J\}\subseteq (V_{\fin})^*$$ for which $(\epsilon_j,\epsilon_k)=\delta_{jk}$ for all $j,k\in J$. 

Any infinite irreducible locally finite root system $\Delta$ can be realised inside $(V_{\fin}^{\star},(\cdot,\cdot))$ for some suitable set $J$, and is of one of the following types:
\begin{equation*}
\begin{aligned}
A_J&:=\{\epsilon_j-\epsilon_k \ | \ j,k\in J, \ j\neq k\},\\
B_J&:=\{\pm \epsilon_j, \pm(\epsilon_j\pm\epsilon_k) \ | \ j,k\in J, \ j\neq k\},\\
C_J&:=\{\pm 2\epsilon_j, \pm(\epsilon_j\pm\epsilon_k) \ | \ j,k\in J, \ j\neq k\},\\
D_J&:=\{\pm(\epsilon_j\pm\epsilon_k) \ | \ j,k\in J, \ j\neq k\},\\
BC_J&:=\{\pm \epsilon_j, \pm 2\epsilon_j, \pm(\epsilon_j\pm\epsilon_k) \ | \ j,k\in J, \ j\neq k\}.
\end{aligned}
\end{equation*}
Set
$$\widehat{V}:=\RR\times V\times\RR, \quad \widehat{V}_{\fin}:=\RR\times V_{\fin}\times\RR, \quad \widehat{V}^{\star}:=\RR\times (V_{\fin})^*\times\RR\quad\textrm{and}\quad \widehat{V}_{\fin}^{\star}:=\RR\times V_{\fin}^{\star}\times\RR,$$
where we identify $\widehat{V}^{\star}$ with the dual of $\widehat{V}_{\fin}$ by setting 
$$\la(z,h,t):=\la_cz+\la^{0}(h)+\la_dt\quad\textrm{for all $\la=(\la_c,\la^{0},\la_d)\in \widehat{V}^{\star}$ and $(z,h,t)\in \widehat{V}_{\fin}$.}$$ In other words, the superscript $\star$ (resp. its absence) indicates that we are considering triples $(z,h,t)$ with $h$ of the form $h=\sum_{j\in J}{h_j\epsilon_j}$ (resp.~$h=\sum_{j\in J}{h_je_j}$) for some $h_j\in\RR$, and the subscript \emph{fin} indicates that we in addition assume that only finitely many $h_j$ are nonzero.

Any infinite irreducible reduced locally affine root system can be realised inside $$V_{\fin}^{\star}\times\ZZ\approx \{0\}\times V_{\fin}^{\star}\times\ZZ\subseteq\widehat{V}_{\fin}^{\star}$$ for some suitable set $J$, and is of one of the following types:
\begin{equation*}
\begin{aligned}
X_J^{(1)}&:=X_J\times\ZZ\quad\textrm{for $X\in\{A,B,C,D\}$},\\
B_J^{(2)}&:=(B_J\times 2\ZZ)\cup \big(\{\pm\epsilon_j \ | \ j\in J\}\times (2\ZZ+1)\big),\\
C_J^{(2)}&:=(C_J\times 2\ZZ)\cup \big(D_J\times (2\ZZ+1)\big),\\
BC_J^{(2)}&:=(B_J\times 2\ZZ)\cup \big(BC_J\times (2\ZZ+1)\big).
\end{aligned}
\end{equation*}
Let $\widehat{\Delta}\subseteq \widehat{V}_{\fin}^{\star}$ be one of the above locally affine root systems $X_J^{(1)}$ or $X_J^{(2)}$, where $\Delta\subseteq V_{\fin}^{\star}$ is the corresponding locally finite root system of type $X_J$. Thus $\widehat{\Delta}\subseteq \{0\}\times \Delta\times\ZZ$. We set $$\Delta_0=\{\alpha\in\Delta \ | \ (0,\alpha,0)\in\widehat{\Delta}\}.$$ Then
$\Delta_0=\Delta$, unless $\widehat{\Delta}$ is of type $BC_J^{(2)}$, in which case $\Delta_0$ is a root subsystem of $\Delta$ of type $B_J$.

The assignment $\epsilon_j\mapsto e_j$, $j\in J$, induces an $\RR$-linear map 
$\sharp\co (V_{\fin})^*\to V:\mu\mapsto\mu^{\sharp}$
(which is the identity if one identifies $(V_{\fin})^*$ with $\RR^J$). For any $\alpha\in \Delta$, we let $$\check{\alpha}:=\frac{2}{(\alpha,\alpha)}\alpha^{\sharp}\in V_{\fin}$$
denote the \emph{coroot} of $\alpha$. We will also view $\alpha$ as a linear functional on $V$ (and not just on $V_{\fin}$) by setting 
$$\alpha(h):=\langle \alpha^{\sharp},h\rangle\quad\textrm{for all $h\in V$.}$$
Finally, denoting by $\kappa\co \widehat{V}_{\fin}\times\widehat{V}\to\RR$ the bilinear form defined by $$\kappa((z_1,h_1,t_1),(z_2,h_2,t_2))=\langle h_1,h_2\rangle -z_1t_2-z_2t_1,$$
we can extend the map $\sharp\co (V_{\fin})^*\to V$ to an $\RR$-linear map $$\sharp\co \widehat{V}^{\star}\to \widehat{V}:\mu=(\mu_c,\mu^{0},\mu_d)\mapsto \mu^{\sharp}=(-\mu_d,(\mu^{0})^{\sharp},-\mu_c)$$
characterised by the property that $$\mu((z,h,t))=\kappa((z,h,t),\mu^{\sharp})\quad\textrm{for all $(z,h,t)\in \widehat{V}_{\fin}$.}$$
The \emph{coroot} of $(\alpha,n)\in\widehat{\Delta}$ is given by
$$(\alpha,n)^{\vee}=\frac{2}{(\alpha,\alpha)}(0,\alpha,n)^{\sharp}=\big(\tfrac{-2n}{(\alpha,\alpha)},\check{\alpha},0\big).$$

\begin{remark}\label{remark:PECisom1}
As announced in the introduction, we will use the main results of \cite{PECisom} to characterise the positive energy highest weight representations of arbitrary affinisations of simple Hilbert--Lie algebras. We wish to attract the attention of the reader to the fact that the choices of parametrisation of these affinisations that we made in the present paper slightly differ from the choices made in \cite{PECisom}, and this in two respects (these choices being better suited for each of the papers). We now explain these differences in more detail and relate the notation introduced so far to the context of affinisations of simple Hilbert--Lie algebras.

Consider, as in \cite[\S 3.1]{PECisom}, a simple Hilbert--Lie algebra $\kk$, an automorphism $\varphi\in\Aut(\kk)$ of finite order $N_{\varphi}$, and a maximal abelian subalgebra $\ttt_0$ of $\kk^{\varphi}$. For $N\in\NN$, we denote by $$\LLL_{\varphi,N}(\kk):=\{\xi\in C^{\infty}(\RR,\kk) \ | \ \xi(t+\tfrac{2\pi}{N})=\varphi\inv(\xi(t)) \ \forall t\in\RR\}$$
the $\varphi$-twisted loop algebra over $\kk$, whose elements are periodic smooth functions of period $2\pi N_{\varphi}/N$. In the present paper, we made the choice $N=N_{\varphi}$, that is, we consider $2\pi$-periodic functions. This is the first difference with \cite{PECisom}, where the choice $N=1$ is made. However, these two choices yield isomorphic objects: explicit isomorphisms were provided in \cite[Remark~4.3]{PECisom}.

Let now $N\in\{1,N_{\varphi}\}$, and let $D_0(\xi)=\xi'$ be the standard derivation of $\LLL_{\varphi,N}(\kk)$. 
Assume that the corresponding affinisation 
$$\g=\g_N=(\RR\oplus_{\omega_{D_{0}}}\LLL_{\varphi,N}(\kk))\rtimes_{\widetilde{D}_{0}}\RR$$
of $\kk$ is standard, with set of compact roots $\widehat{\Delta}=\Delta(\g,\ttt_0^e)_c\subseteq i(\ttt_0^e)^*$ with respect to the Cartan subalgebra
$$\ttt_0^e=\RR\oplus\ttt_0\oplus\RR$$
of one of the types $A_J^{(1)}$, $B^{(1)}_J$, $C^{(1)}_J$, $D_J^{(1)}$, $B_J^{(2)}$, $C_J^{(2)}$ and $BC_J^{(2)}$ described above (see \cite[\S 3.4 and \S 3.5]{PECisom}). Set 
$\bc=(i,0,0)\in i\ttt_0^e$ and $\bd=(0,0,-i)\in i\ttt_0^e$.

The second difference is purely notational, and concerns the identification of the Cartan subalgebra $\ttt_0^e$ of $\g$ (or rather, of $i\ttt_0^e\subseteq \g_{\CC}$) with the space of triples $\RR\times i\ttt_0\times\RR$: the description of $\widehat{\Delta}$ inside $\sppan_{\ZZ}\{\epsilon_j \ | \ j\in J\}\times\ZZ$ (which is the same in both papers) yields identifications
\begin{equation}
V_{\fin}^{(2)}\approx i\ttt_0\quad \textrm{and}\quad \widehat{V}_{\fin}^{(2)}\stackrel{\sim}{\to} i\ttt_0^e:(z,h,t)\mapsto z\bc+h+t\bd,
\end{equation}
where $V_{\fin}^{(2)}$ denotes the Hilbert space completion of $V_{\fin}$ and $\widehat{V}_{\fin}^{(2)}:=\RR\times V_{\fin}^{(2)}\times\RR$.
The $\RR$-linear map $\sharp\co (V_{\fin})^*\to V$ then coincides with the map $\sharp\co i\ttt_0^*\to i\widehat{\ttt_0}$ defined in \cite[\S 2.3 and \S 3.1]{PECisom}, while its extension $\sharp\co \widehat{V}^{\star}\to \widehat{V}$ coincides with the map $\sharp\co i(\ttt^e_0)^*\to i\widehat{\ttt^e_0}$ defined in \cite[\S 7.2]{PECisom}. Similarly, the bilinear form $\kappa$ on $\widehat{V}_{\fin}$ (or rather, its extension to $\widehat{V}_{\fin}^{(2)}$) coincides with the restriction to $i\ttt_0^e$ of the hermitian extension of the bilinear form $\kappa$ defined in \cite[\S 3.4]{PECisom}. 
Finally, note that, following \cite[\S 3.4]{PECisom}, the root $(\alpha,n)\in\widehat{\Delta}$ of the affinisation $\g_1$ of $\kk$ satisfies
$$(\alpha,n)(h)=\alpha(h)\quad\textrm{for $h\in i\ttt_0$}, \quad (\alpha,n)(\bc)=0 \quad\textrm{and}\quad (\alpha,n)(\bd)=n/N_{\varphi}.$$
Hence the reparametrisation provided by \cite[Remark~4.3]{PECisom} of $(\alpha,n)$ as a root of $\g_{N_{\varphi}}$ yields that
$$(\alpha,n)(z,h,t)=(0,\alpha,n)(z,h,t)=\alpha(h)+nt\quad\textrm{for all $(z,h,t)\in \widehat{V}_{\fin}$},$$
in accordance with our identification of $\widehat{V}^{\star}$ with the dual of $\widehat{V}_{\fin}$.
\end{remark}

\subsection{The Weyl group of \texorpdfstring{$\widehat{\Delta}$}{Delta}}\label{subsection:Weyl_group}
%\subsection{The Weyl group of $\widehat{\Delta}$}\label{subsection:Weyl_group}
Let $S_J$ denote the set of bijections of $J$, which we view as a subgroup of $\GL(V)$ with $w\in S_J$ acting as $w(e_j):=e_{w(j)}$. The \emph{support} of a permutation $w\in S_J$ with fixed-point set $I\subseteq J$ is the set $J\setminus I$. We denote by $S_{(J)}\subseteq S_J$ the subgroup of permutations of $S_J$ with finite support, and we view it as a subgroup of either $\GL(V)$ or $\GL(V_{\fin})$.

We also let $\{\pm 1\}^J\subseteq\RR^J$ act linearly on $V=\RR^J$ by componentwise left multiplication: $\sigma(e_j)=\sigma_je_j$ for $\sigma=(\sigma_j)_{j\in J}\in \{\pm 1\}^J$. The \emph{support} of an element $\sigma=(\sigma_j)_{j\in J}\in \{\pm 1\}^J$ is the set $\{j\in J \ | \ \sigma_j=-1\}$. We denote by $\{\pm 1\}^{(J)}$ (resp. $\{\pm 1\}_2^{(J)}$) the set of elements of $\{\pm 1\}^J$ with finite (resp. finite and even) support, and we again view $\{\pm 1\}^{(J)}$ and $\{\pm 1\}_2^{(J)}$ as subgroups of either $\GL(V)$ or $\GL(V_{\fin})$.

We recall from \cite[\S 2.2]{PEClocfin} that for $X\in\{A,B,C,D,BC\}$, the Weyl group $\WW(X_J)$ of type $X_J$ admits the following description:
\begin{equation*}
\begin{aligned}
\WW(A_J)&=S_{(J)},\\
\WW(B_J)&=\WW(C_J)=\WW(BC_J)=\{\pm 1\}^{(J)}\rtimes S_{(J)},\\
\WW(D_J)&=\{\pm 1\}_2^{(J)}\rtimes S_{(J)}.
\end{aligned}
\end{equation*}

We denote by $\WW=\WW(\Delta_0)$ the Weyl group corresponding to $\Delta_0$, viewed as a subgroup of either $\GL(V)$ or $\GL(V_{\fin})$. Thus $\WW=\WW(X_J)$ if $\widehat{\Delta}$ is of type $X_J^{(1)}$ or $X_J^{(2)}$ (see \S\ref{section:preliminaries:affine}).
Similarly, we denote by $\widehat{\WW}=\widehat{\WW}(X)$ the Weyl group corresponding to $\widehat{\Delta}$, where $X=X_J^{(1)}$ or $X_J^{(2)}$ is the type of $\widehat{\Delta}$, and
we view $\widehat{\WW}$ as a subgroup of either $\GL(\widehat{V})$ or $\GL(\widehat{V}_{\fin})$.

The group $\widehat{\WW}$ is generated by the set of reflections $\big\{r_{(\alpha,n)} \ | \ (\alpha,n)\in\widehat{\Delta}, \ \alpha\neq 0\big\}$, where
\begin{equation}
r_{(\alpha,n)}(z,h,t)= (z,h,t)-(\alpha(h)+nt)\big(\tfrac{-2n}{(\alpha,\alpha)},\check{\alpha},0\big) \quad\textrm{for all $(z,h,t)\in \widehat{V}$.}
\end{equation}
We view $\WW$ as a subgroup of $\widehat{\WW}$, using the identification $$\WW\cong\langle r_{(\alpha,0)} \ | \ \alpha\in\Delta_0\rangle\subseteq\widehat{\WW}.$$
For each $x\in V_{\fin}$, we define the linear automorphism $\tau_x=\tau(x)$ of $\widehat{V}$ (resp. of $\widehat{V}_{\fin}$) by
\begin{equation}\label{eqn:tau}
\tau_x(z,h,t)=\Big( z+\langle x,h\rangle+\frac{t\langle x,x\rangle}{2},h+tx,t\Big)\quad\textrm{for all $(z,h,t)\in \widehat{V}$}.
\end{equation}
Then $\tau_{x_1}\tau_{x_2}=\tau_{x_1+x_2}$ for all $x_1,x_2\in V_{\fin}$. Moreover, 
$r_{(\alpha,0)}r_{(\alpha,n)}=\tau_{n\check{\alpha}}$ for all $\alpha\in \Delta$ and $n\in\ZZ$. Since for any $\alpha\in\Delta$ there exists some $\beta\in\Delta_0$ such that $r_{(\alpha,0)}=r_{(\beta,0)}$ (as can be seen from a quick inspection of the locally affine root systems), we thus get a semi-direct decomposition 
$$\widehat{\WW}=\tau(\TTT)\rtimes \WW\subseteq\GL(V),$$
where $\TTT$ is the additive subgroup of $V_{\fin}$ generated by $\big\{n\check{\alpha} \ | \ (\alpha,n)\in\widehat{\Delta}\big\}$ (see \cite[\S 3.4]{convexhull} or else \cite[\S 3.6]{PECisom}). Set $Q:=\bigoplus_{j\in J}{\ZZ e_j}\subseteq V_{\fin}$. For $\widehat{\WW}=\widehat{\WW}(X)$ of type $X$, one can then describe the corresponding lattice $\TTT=\TTT(X)\subseteq Q$ of type $X$ as follows (see \cite[Proposition~3.12]{convexhull}): 
\begin{equation*}
\begin{aligned}
\TTT(A_J^{(1)}) & =\Big\{\sum_{j\in J}{n_je_j}\in Q \ | \ \sum_{j\in J}{n_j}=0\Big\},\\
\TTT(B^{(1)}_J) = \TTT(D^{(1)}_J) = \TTT(C^{(2)}_J) 
& = \Big\{\sum_{j\in J}{n_je_j}\in Q \ | \ \sum_{j\in J}{n_j}\in 2\ZZ\Big\},\\
\TTT(B_J^{(2)}) & =\Big\{\sum_{j\in J}{n_je_j}\in Q \ | \ n_j\in 2\ZZ \ \forall j\in J\Big\}, 
\\
\TTT(C^{(1)}_J) = \TTT(BC^{(2)}_J)  & = Q = \ZZ^{(J)}.\\
\end{aligned}
\end{equation*}
Thus any element $\widehat{w}\in\widehat{\WW}$ can be uniquely written as a product $\widehat{w}=\tau_x\sigma w$ for some $x\in\TTT$, some $\sigma\in\{\pm 1\}^{(J)}$ and some $w\in S_{(J)}$. For $x=\sum_{j\in J}{n_je_j}\in\TTT$, we call the subset $\{j\in J \ | \ n_j\neq 0\}$ of $J$ the \emph{support} of $x$. 

\begin{remark}
The Weyl group $\widehat{\WW}$ may also be viewed, as in the introduction, as a subgroup of $\GL(\widehat{V}^{\star})$ using the bijection $\sharp\co \widehat{V}^{\star}\to \widehat{V}$, or in other words, by requiring that $$(\widehat{w}.\mu)^{\sharp}=\widehat{w}.\mu^{\sharp}\quad\textrm{for all $\widehat{w}\in\widehat{\WW}$ and $\mu\in \widehat{V}^{\star}$.}$$
\end{remark}

\subsection{The positive energy condition}\label{subsection:PEC}
In the sequel, we fix some $$\la=(\la_c,\la^{0},\la_d)\in\RR\times \RR^{J}\times\RR\approx \widehat{V}^{\star}\quad\textrm{and}\quad\chi=(\chi_c,\chi^{0},\chi_d)\in\RR\times \RR^{J}\times\RR=\widehat{V},$$
and we write 
$$\la^{0}=(\la_j)_{j\in J}=\sum_{j\in J}{\la_j\epsilon_j}\in\RR^{J}\approx (V_{\fin})^*\quad\textrm{and}\quad \chi^{0}=(d_j)_{j\in J}=\sum_{j\in J}{d_je_j}\in\RR^{J}=V.$$

\begin{definition}\label{def:PEC}
We say that the triple $(J,\la,\chi)$ \emph{satisfies the positive energy condition (PEC) for $\widehat{\WW}$} if the set $\la(\widehat{\WW}.\chi-\chi)$ is bounded from below. We moreover say that $(J,\la,\chi)$ is of \emph{minimal energy for $\widehat{\WW}$} if $\inf\big(\la\big(\widehat{\WW}.\chi-\chi\big)\big)=0$, that is, if $\la(\widehat{w}.\chi-\chi)\geq 0$ for all $\widehat{w}\in\widehat{\WW}$. 

Note that $\widehat{w}.\chi-\chi\in \widehat{V}_{\fin}$ for any $\widehat{w}\in\widehat{\WW}$, so that $\la(\widehat{w}.\chi-\chi)$ is defined. Indeed, writing $\widehat{w}=\tau_xw$ for some $x\in\TTT$ and $w\in\WW$, we have
\begin{equation}\label{eqn:PECwelldef}
\begin{aligned}
\widehat{w}.\chi-\chi &=\tau_xw.(\chi_c,\chi^{0},\chi_d)-(\chi_c,\chi^{0},\chi_d)\\
&= \tau_x\bigg(\chi_c,\sum_{j\in J}{d_jw(e_j)},\chi_d\bigg)-\bigg(\chi_c,\sum_{j\in J}{d_je_j},\chi_d\bigg)\\
&= \bigg(\chi_c+\sum_{j\in J}{d_j\langle w(e_j),x\rangle}+\frac{\chi_d\langle x,x\rangle}{2},\sum_{j\in J}{d_jw(e_j)}+\chi_dx,\chi_d\bigg)-\bigg(\chi_c,\sum_{j\in J}{d_je_j},\chi_d\bigg)\\
&= \bigg(\sum_{j\in J}{d_j\langle w(e_j),x\rangle}+\frac{\chi_d\langle x,x\rangle}{2},\sum_{j\in J}{d_j(w(e_j)-e_j)}+\chi_dx,0\bigg)\in\widehat{V}_{\fin}.\\
\end{aligned}
\end{equation}
\end{definition}

\begin{remark}\label{remark:lachi_chila}
Any character $\chi\co\ZZ[\widehat{\Delta}]\to\RR$ can be identified with an element of $\widehat{V}$ by requiring that 
$$\chi(\mu)=\kappa(\mu^{\sharp},\chi)\quad\textrm{for all $\mu\in \ZZ[\widehat{\Delta}]$.}$$
With this identification, and since the Weyl group $\widehat{\WW}$ preserves the bilinear form $\kappa$ (see \cite[\S 3.6]{PECisom}), we deduce for all $\widehat{w}\in\widehat{\WW}$ that
$$\chi(\widehat{w}.\la-\la)=\kappa(\widehat{w}.\la^{\sharp}-\la^{\sharp},\chi)=\kappa(\widehat{w}\inv.\chi-\chi,\la^{\sharp})=\la(\widehat{w}\inv.\chi-\chi).$$
In particular, the conditions $\inf \big(\chi\big(\widehat{\WW}.\la-\la\big)\big)>-\infty$ and $\inf \big(\chi\big(\widehat{\WW}.\la-\la\big)\big)=0$ are respectively equivalent to the conditions $\inf \big(\la\big(\widehat{\WW}.\chi-\chi\big)\big)>-\infty$ and $\inf \big(\la\big(\widehat{\WW}.\chi-\chi\big)\big)=0$. Thus the notions of positive energy and minimal energy introduced in Definition~\ref{def:PEC} indeed coincide with the corresponding notions from the introduction.
\end{remark}

Let $\widehat{w}\in\widehat{\WW}$. Write $\widehat{w}=\tau_x\sigma w\inv$ for some $x=\sum_{j\in J}{n_je_j}\in\TTT$, some $\sigma=(\sigma_j)_{j\in J}\in \{\pm 1\}^{(J)}$ and some $w\in S_{(J)}$ (see \S\ref{subsection:Weyl_group}).
It then follows from (\ref{eqn:PECwelldef}) that
\begin{equation}\label{eqn:first_eqn}
\begin{aligned}
\la(\widehat{w}.\chi-\chi)&=\la(\tau_x\sigma w\inv.\chi-\chi) \\
&=\la\bigg(\sum_{j\in J}{d_j\langle \sigma_{w\inv(j)}e_{w\inv(j)},x\rangle}+\frac{\chi_d\langle x,x\rangle}{2},\sum_{j\in J}{d_j(\sigma_{w\inv(j)}e_{w\inv(j)}-e_j)}+\chi_dx,0\bigg)\\
&=\frac{\la_c\chi_d}{2}\langle x,x\rangle+\la_c\sum_{j\in J}{\sigma_jd_{w(j)}\langle e_j,x\rangle}+\chi_d\la^{0}(x)+\sum_{j\in J}{d_{w(j)}\la^{0}(\sigma_je_{j}-e_{w(j)})}\\
&=\frac{\la_c\chi_d}{2}\sum_{j\in J}{n_j^2}+\la_c\sum_{j\in J}{n_j\sigma_jd_{w(j)}}+\chi_d\sum_{j\in J}{n_j\la_j}+\sum_{j\in J}{\la_j(\sigma_jd_{w(j)}-d_j)}.\\
\end{aligned}
\end{equation}

\subsection{PEC for locally finite root systems}
The concept of PEC for a triple $(J,\la,\chi)$ extends the concept of PEC for the triple $(J,\la^{0},\chi^{0})$ introduced in \cite[\S 2.3]{PEClocfin}. We recall that the triple $(J,\la^{0},\chi^{0})$ is said to \emph{satisfy the PEC for $\WW$} if $\la^{0}(\WW.\chi^{0}-\chi^{0})$ is bounded from below. A complete characterisation of such triples (with some suitable finiteness assumption on $\la^{0}$), analoguous to the one we give in this paper, was provided in \cite{PEClocfin}. We now briefly review this characterisation, as it will be the starting point of our study of the PEC for triples $(J,\la,\chi)$: if $(J,\la,\chi)$ satisfies the PEC for $\widehat{\WW}$, then $(J,\la^{0},\chi^{0})$ satisfies the PEC for $\WW$.

Let $\la,\chi$ be as in \S\ref{subsection:PEC}. We first recall some notation and terminology from \cite[Section~3]{PEClocfin}. Define the functions $$D\co J\to \RR:j\mapsto d_j \quad\textrm{and}\quad\Lambda\co J\to\RR:j\mapsto\la_j,$$ as well as the sets $J_n:=\Lambda\inv(n)$ for each $n\in\RR$. For $r\in\RR$, we also set
$$J_n^{>r}:=\{j\in J_n \ | \ d_j>r\}\quad\textrm{and}\quad J_n^{<r}:=\{j\in J_n \ | \ d_j<r\}.$$

\begin{definition}
We call $\la^{0}$ \emph{finite} if the subset $\Lambda(J)$ of $\RR$ is finite.
\end{definition}

\begin{definition}
Let $r,n\in\RR$. A subset $I\subseteq J$ of the form $J_n^{>r}$ or $J_n^{<r}$ is said to be \emph{summable} if $\sum_{j\in I}{|d_j-r|}<\infty$. 
\end{definition}

\begin{definition}\label{definition:accumulation_point}
Let $I\subseteq J$. We call $r\in\RR$ an \emph{accumulation point for $I$} if either $r$ is an accumulation point for $D(I)$, or if $D(I')=\{r\}$ for some infinite subset $I'\subseteq I$.
\end{definition}

\begin{definition}[{\cite[Definition~5.2]{PEClocfin}}]\label{definition:cones_fin}
For a set $J$ and a tuple $\la^{0}=(\la_j)_{j\in J}\in\RR^{J}$, we define the following 
cones in $\RR^{J}$:
\begin{equation*}
\begin{aligned}
C_{\min}(\la^{0},A_J)&=\{(d_j)_{j\in J}\in \RR^J\ | \  \forall i,j\in J: \ \la_i<\la_j\implies  d_i\geq d_j\},\\
C_{\min}(\la^{0},B_J)&=\{(d_j)_{j\in J}\in \RR^J \ | \ \forall j\in J: \ \la_jd_j\leq 0\quad\textrm{and}\quad \forall i,j\in J: \ |\la_i|<|\la_j|\implies  |d_i|\leq |d_j|\}.\\
\end{aligned}
\end{equation*}
We also define the vector subspace $\ell^1(J):=\{(d_j)_{j\in J}\in\RR^{J}  \ | \ \sum_{j\in J}{|d_j|}<\infty\}$ of $\RR^{J}$.
\end{definition}

We next recall the characterisation of triples $(J,\la^{0},\chi^{0})$ of minimal energy.
\begin{prop}[{\cite[Proposition~5.3]{PEClocfin}}]\label{prop:minimality}
Let $X\in\{A,B\}$ and set $\WW=\WW(X_J)$. For a triple $(J,\la^{0},\chi^{0})$, the following assertions are equivalent:
\begin{enumerate}
\item
$\inf\la^{0}(\WW.\chi^{0}-\chi^{0})=0$.
\item
$\chi^{0}\in C_{\min}(\la^{0},X_J)$.
\end{enumerate}
\end{prop}

Note that for $X\in\{C,D,BC\}$, if we denote by $C_{\min}(\la^{0},X_J)$ the set of tuples $\chi^{0}\in\RR^J$ for which $\inf\big(\la^{0}(\WW(X_J).\chi^{0}-\chi^{0})\big)=0$, then the inclusions $$\WW(A_J)\subseteq \WW(D_J)\subseteq \WW(B_J)=\WW(C_J)=\WW(BC_J)$$ of Weyl groups imply that
$$C_{\min}(\la^{0},BC_J)= C_{\min}(\la^{0},C_J)=C_{\min}(\la^{0},B_J)\subseteq C_{\min}(\la^{0},D_J)\subseteq C_{\min}(\la^{0},A_J).$$

Finally, we recall the characterisation of triples $(J,\la^{0},\chi^{0})$ of positive energy.
\begin{theorem}[{\cite[Theorems~5.10 and 5.12]{PEClocfin}}]\label{theorem:charact_PEC_locally_finite}
Let $J$ be a set, and let $\la^{0}=(\la_j)_{j\in J}$ and $\chi^{0}=(d_j)_{j\in J}$ be elements of $\RR^J$. Assume that $\la^{0}$ is finite. Then for $X\in\{A,B\}$, the following assertions are equivalent:
\begin{enumerate}
\item
$(J,\la^{0},\chi^{0})$ satisfies the PEC for $\WW(X_J)$.
\item
$\chi^{0}\in C_{\min}(\la^{0},X_J)+\ell^1(J)$.
\end{enumerate}
\end{theorem}

\begin{lemma}[{\cite[Lemma~5.8]{PEClocfin}}]\label{lemma:BCD}
Let $X\in\{B,C,D,BC\}$, and assume that $\la^{0}$ is finite. Then $(J,\la^{0},\chi^{0})$ satisfies the PEC for $\WW(X_J)$ if and only if it satisfies the PEC for $\WW(B_J)$.
\end{lemma}

\section{Reduction steps}
For the rest of this paper, $\la=(\la_c,\la^{0},\la_d)$ and $\chi=(\chi_c,\chi^{0},\chi_d)$ will always denote elements of $\RR\times\RR^J\times\RR$ as in \S\ref{subsection:PEC}. 
In order to characterise the PEC for $(J,\la,\chi)$, we will need to make some finiteness assumption on $\la$.
\begin{definition}
We call $\la$ \emph{$\ZZ$-discrete} if $\la_c\neq 0$ and if the set of cosets $\big\{\tfrac{\la_j}{\la_c}+\ZZ \ | \ j\in J\big\}$ is finite. 
\end{definition}
A more illuminating formulation of this $\ZZ$-discreteness assumption will be given in \S\ref{subsection:TI} below (see Remark~\ref{remark:Z-discrete_alt}). Note that the case $\la_c= 0$ can be easily dealt with (see \S\ref{section:trivial_cases} below). On the other hand, if we want $\la$ to correspond to some highest weight of an integrable highest weight module as in Theorem~\ref{thmintro:A}, then we need $\la$ to be \emph{integral} with respect to $\widehat{\Delta}$, in the sense that $\la((\alpha,n)^{\vee})\in\ZZ$ for all $(\alpha,n)\in\widehat{\Delta}$. The following lemma then shows that, for representation theoretic purposes, we may safely assume $\la$ to be $\ZZ$-discrete.

\begin{lemma}\label{lemma:integral-Zdiscrete}
Assume that $\la_c\neq 0$ and that $\la$ is integral with respect to $\widehat{\Delta}$. Then $\la$ is $\ZZ$-discrete. 
\end{lemma}
\begin{proof}
Let $i,j\in J$ with $i\neq j$, and for each $n\in\ZZ$, set $\gamma_n:=(0,\epsilon_i-\epsilon_j,n)\in \widehat{V}^{\star}_{\fin}$. Then $\gamma_n\in\widehat{\Delta}$ for infinitely many values of $n\in\ZZ$. Since $\gamma^{\vee}_n=(-n,e_i-e_j,0)$, the integrality condition on $\la$ implies that 
$$\la(\gamma^{\vee}_n)=-n\la_c+\la_i-\la_j\in\ZZ$$
for at least two distinct values of $n\in\ZZ$, so that $\la_c\in\QQ$. Write $\la_c=m/p$ for some nonzero $m,p\in\ZZ$. Then
$$\tfrac{\la_i}{\la_c}-\tfrac{\la_j}{\la_c}\in \ZZ+\tfrac{1}{\la_c}\ZZ
\subseteq \ZZ +\tfrac{1}{m}\ZZ  \subseteq \tfrac{1}{m}\ZZ 
\quad\textrm{for all $i,j\in J$.}$$
Fixing some $i_0\in J$, we deduce that 
$$\tfrac{\la_j}{\la_c}+\ZZ\in\big\{\tfrac{\la_{i_0}}{\la_c}+\tfrac{s}{m}+\ZZ \ \big| \ s=0,1,\dots,m-1\big\}\quad\textrm{for all $j\in J$.}$$
Hence the set of cosets 
$\big\{\tfrac{\la_j}{\la_c}+\ZZ \ \big| \ j\in J\big\}$
is finite, so that $\la$ is $\ZZ$-discrete, as desired.
\end{proof}

We begin our study of the PEC for the triple $(J,\la,\chi)$ by some reduction steps.

\subsection{\texorpdfstring{The case $\la_c\chi_d=0$}{The degenerate case}}\label{section:trivial_cases}
%\subsection{The case $\la_c\chi_d=0$}\label{section:trivial_cases}
In this paragraph, we investigate the PEC for the triple $(J,\la,\chi)$ in case $\la_c\chi_d=0$.

\begin{lemma}\label{lemma:trivial_case00}
Assume that $\la_c=\chi_d=0$. Then $(J,\la,\chi)$ satisfies the PEC for $\widehat{\WW}$ if and only if $(J,\la^{0},\chi^{0})$ satisfies the PEC for $\WW$.
\end{lemma}
\begin{proof}
This readily follows from (\ref{eqn:first_eqn}) in Section~\ref{subsection:PEC}.
\end{proof}

\begin{lemma}
Assume that $\la_c=0$ but $\chi_d\neq 0$. Then $(J,\la,\chi)$ satisfies the PEC for $\widehat{\WW}$ if and only if one of the following holds:
\begin{enumerate}
\item
$\widehat{\WW}=\widehat{\WW}(A_J^{(1)})$ and $\la^{0}$ is constant.
\item
$\widehat{\WW}\neq\widehat{\WW}(A_J^{(1)})$ and $\la^{0}=0$.
\end{enumerate}
\end{lemma}

\begin{proof}
Assume that $(J,\la,\chi)$ satisfies the PEC for $\widehat{\WW}$, and set 
$\sigma=w=1$ in (\ref{eqn:first_eqn}). Then 
$$\Big\{ \chi_d\sum_{j\in J}{n_j\la_j} \ \big| \ \sum_{j\in J}{n_je_j}\in\TTT \Big\}$$
must be bounded from below. This is easily seen to imply (1) or (2), depending on the type of $\widehat{\WW}$. The converse is an easy consequence of (\ref{eqn:first_eqn}).
\end{proof}

\begin{lemma}
Assume that $\chi_d=0$ but $\la_c\neq 0$. Then $(J,\la,\chi)$ satisfies the PEC for $\widehat{\WW}$ if and only if one of the following holds:
\begin{enumerate}
\item
$\widehat{\WW}=\widehat{\WW}(A_J^{(1)})$ and $\chi^{0}$ is constant.
\item
$\widehat{\WW}\neq\widehat{\WW}(A_J^{(1)})$ and $\chi^{0}=0$.
\end{enumerate}
\end{lemma}
\begin{proof}
Assume that $(J,\la,\chi)$ satisfies the PEC for $\widehat{\WW}$, and set $\sigma=w=1$ in (\ref{eqn:first_eqn}). Then 
$$\Big\{ \la_c\sum_{j\in J}{n_jd_{j}} \ \big| \ \sum_{j\in J}{n_je_j}\in\TTT\Big\}$$
must be bounded from below. This is easily seen to imply (1) or (2), depending on the type of $\widehat{\WW}$. The converse is an easy consequence of (\ref{eqn:first_eqn}).
\end{proof}

\subsection{\texorpdfstring{The case $\la_c\chi_d\neq 0$}{The non-degenerate case}}\label{subsection:nontrivial_case}
%\subsection{The case $\la_c\chi_d\neq 0$}\label{subsection:nontrivial_case}
In view of \S\ref{section:trivial_cases}, we may now assume that $\la_c\chi_d\neq 0$. 
We begin with two simple observations.
\begin{lemma}\label{lemma:SC0}
Assume that $(J,\la,\chi)$ satisfies the PEC for $\widehat{\WW}$ and that $\la_c\chi_d\neq 0$. Then $\la_c\chi_d>0$.
\end{lemma}
\begin{proof}
Fix some $i,j\in J$ with $i\neq j$, and for each $n\in\NN$, consider the element $x_n=2n(e_i-e_j)\in\TTT$. Setting $\sigma=w=1$ and $x=x_n$ in (\ref{eqn:first_eqn}), the PEC then implies that  
$$ \{4\la_c\chi_dn^2+2\la_c(d_i-d_j)n+2\chi_d(\la_i-\la_j)n \ | \ n \in \NN\} $$
must be bounded from below, yielding the claim.
\end{proof}

\begin{lemma}\label{lemma:simpleobs2}
Assume that $\la_c\chi_d>0$. Then the following are equivalent:
\begin{enumerate}
\item
$(J,\la,\chi)$ satisfies the PEC for $\widehat{\WW}$.
\item
$(J,\la_{\mathrm{st}},\chi_{\mathrm{st}}):=(J,(1,\tfrac{\la^{0}}{\la_c},0),(0,\tfrac{\chi^{0}}{\chi_d},1))$ satisfies the PEC for $\widehat{\WW}$.
\end{enumerate}
More precisely, $\inf\la\big(\widehat{\WW}.\chi-\chi\big)=\la_c\chi_d\cdot \inf\la_{\mathrm{st}}\big(\widehat{\WW}.\chi_{\mathrm{st}}-\chi_{\mathrm{st}}\big)$.
\end{lemma}
\begin{proof}
This readily follows from (\ref{eqn:first_eqn}).
\end{proof}

Thus, in order to investigate the PEC for $(J,\la,\chi)$, we may safely make the following normalisation.
\begin{convention}\label{convention:normalisation}
From now on, unless otherwise stated, we assume that $\la_c=\chi_d=1$ and that $\la_d=\chi_c=0$, so that
$$\la=(1,\la^{0},0)\quad\textrm{and}\quad \chi=(0,\chi^{0},1).$$
\end{convention}

Given $\widehat{w}\in\widehat{\WW}$, which we write as $\widehat{w}=\tau_x\sigma w\inv$ for some $x=\sum_{j\in J}{n_je_j}\in\TTT$, some $w\in S_{(J)}$ and some $\sigma=(\sigma_j)_{j\in J}\in \{\pm 1\}^{(J)}$, we can now rewrite (\ref{eqn:first_eqn}) as
\begin{equation}\label{eqn:second_eqn}
\begin{aligned}
\la(\widehat{w}.\chi-\chi)&=\frac{1}{2}\sum_{j\in J}{n_j^2}+\sum_{j\in J}{n_j\sigma_jd_{w(j)}}+\sum_{j\in J}{n_j\la_j}+\sum_{j\in J}{\la_j(\sigma_jd_{w(j)}-d_j)}\\
&=\frac{1}{2}\sum_{j\in J}{\Big((n_j+\la_j+\sigma_jd_{w(j)})^2-(\la_j+d_j)^2\Big)}.
\end{aligned}
\end{equation}

\subsection{Translation invariance}\label{subsection:TI}

We introduce the following notation.
Given $x\in\RR$, we set $$[x]:=\floor{x+\tfrac{1}{2}}\quad\textrm{and}\quad \fr{x}:=x-[x]\in [-\tfrac{1}{2},\tfrac{1}{2}).$$ In other words, $[x]$ is the closest integer to $x$, with the choice $[x]=\ceil{x}$ if $x\in\frac{1}{2}+\ZZ$. Given a tuple $\mm=(m_j)_{j\in J}\in\RR^J$, we then define the tuples 
$$[\mm]:=([m_j])_{j\in J}\in\ZZ^J\subseteq\RR^J\quad\textrm{and}\quad \fr{\mm}:=(\fr{m_j})_{j\in J}\in [-\tfrac{1}{2},\tfrac{1}{2})^J\subseteq\RR^J.$$
We also set 
$$\la_{\mm}:=(1,\la^{0}-\mm,0)\quad\textrm{and}\quad \chi_{\mm}:=(0,\chi^{0}+\mm,1).$$

\begin{lemma}\label{lemma:translation_invariance1}
Let $\mm=(m_j)_{j\in J}\in\ZZ^J$. If $\widehat{\Delta}$ is of type $B^{(2)}_J$, we moreover assume that $m_j\in 2\ZZ$ for all $j\in J$. Then $\inf\la\big(\widehat{\WW}.\chi-\chi\big)=\inf\la_{\mm}\big(\widehat{\WW}.\chi_{\mm}-\chi_{\mm}\big)$. In particular, the following are equivalent:
\begin{enumerate}
\item
$(J,\la,\chi)$ satisfies the PEC for $\widehat{\WW}$.
\item
$(J,\la_{\mm},\chi_{\mm})$ satisfies the PEC for $\widehat{\WW}$.
\end{enumerate}
\end{lemma}
\begin{proof}
Given $\sigma\in \{\pm 1\}^{(J)}$ and $w\in S_{(J)}$ such that $\sigma w\inv\in \WW$, we will prove that
\begin{equation*}
\inf_{x\in\TTT}\la(\tau_{x}\sigma w\inv .\chi-\chi)=\inf_{x\in\TTT}\la_{\mm}(\tau_{x}\sigma w\inv .\chi_{\mm}-\chi_{\mm}),
\end{equation*}
yielding the claim.
Let $\sigma=(\sigma_j)_{j\in J}\in \{\pm 1\}^{(J)}$ and $w\in S_{(J)}$ be as above, and let $x=\sum_{j\in J}{n_je_j}\in \TTT$. Set $$x'=x+\sum_{j\in J}{(\sigma_{w(j)}m_{w(j)}-m_j)e_j}.$$ One easily checks that $x'\in\TTT$. Moreover, (\ref{eqn:second_eqn}) implies that 
\begin{equation*}
\begin{aligned}
\la_{\mm}(\tau_{x}\sigma w\inv .\chi_{\mm}-\chi_{\mm})&=\frac{1}{2}\sum_{j\in J}{\Big((n_j-m_j+\sigma_{w(j)}m_{w(j)}+\la_j+\sigma_{w(j)}d_{w(j)})^2-(\la_j+d_j)^2\Big)}\\
&=\la(\tau_{x'}\sigma w\inv .\chi-\chi).
\end{aligned}
\end{equation*}
We deduce that 
$$\inf_{x\in\TTT}\la_{\mm}(\tau_{x}\sigma w\inv .\chi_{\mm}-\chi_{\mm})\geq\inf_{x\in\TTT}\la(\tau_{x}\sigma w\inv .\chi-\chi).$$
The same argument with $(J,\la,\chi)$ replaced by $(J,\la_{\mm},\chi_{\mm})$ and $\mm$ replaced by $-\mm$ yields the inequality in the other direction, as desired.
\end{proof}

\begin{remark}\label{remark:Z-discrete_alt}
Note that for $\mm=[\la^{0}]\in\ZZ^J$, the passage from $(J,\la,\chi)$ to $(J,\la_{\mm},\chi_{\mm})$ amounts to replace $\la^{0}$ by $\la^{0}-[\la^{0}]=\fr{\la^{0}}$ and $\chi^{0}$ by $\chi^{0}+[\la^{0}]$. On the other hand, if $\la^{0}=\fr{\la^{0}}$, then $\la$ is $\ZZ$-discrete if and only if $\la^0$ is finite.
\end{remark}

\section{Consequences of the PEC for \texorpdfstring{$\widehat{\WW}(A^{(1)}_J)$}{W(A1J)} and \texorpdfstring{$\widehat{\WW}(C^{(1)}_J)$}{W(C1J)}}
%\section{Consequences of the PEC for $\widehat{\WW}(A^{(1)}_J)$ and $\widehat{\WW}(C^{(1)}_J)$}

\begin{prop}\label{prop:necessary_A_aff}
Assume that $\la$ is $\ZZ$-discrete and that $\la^{0}=\fr{\la^{0}}$. Assume moreover that $(J,\la,\chi)$ satisfies the PEC for $\widehat{\WW}(A^{(1)}_J)$. Then the following hold:
\begin{enumerate}
\item
$D(J)$ is bounded.
\item
If $r_{\min}$ and $r_{\max}$ respectively denote the minimal and maximal accumulation points of $J$, then
\begin{equation*}
r_{\max}-r_{\min}\leq 1.
\end{equation*}
\item
If $r_{\max}-r_{\min}= 1$, then
\begin{equation*}
\sum_{j\in J_+}{(d_j-r_{\max})}<\infty\quad\textrm{and}\quad\sum_{j\in J_-}{(r_{\min}-d_j)}<\infty,
\end{equation*}
where $J_+:=\{j\in J \ | \ d_j>r_{\max}\}$ and $J_-:=\{j\in J \ | \ d_j<r_{\min}\}$.
\end{enumerate}
\end{prop}

\begin{proof} (1) Note first that, by assumption, $\Lambda(J)$ is finite and contained in $[-\tfrac{1}{2},\tfrac{1}{2} \thinspace )$. In particular, to prove that $D(J)$ is bounded, it is sufficient to prove that $D(J_m)$ is bounded for each $m\in\RR$, where 
$J_m := \Lambda^{-1}(m)$. If $J_m$ is finite, there is nothing to prove. Assume now that $J_m$ is infinite, and write it as a disjoint union $J_m=I_{-1}\cup I_0\cup I_1$ of three infinite subsets. Define the tuple $\mm=(m_j)_{j\in J}\in\ZZ^J$ by 
\begin{equation*}
m_j=
\left\{
\begin{array}{ll}
k\quad &\textrm{if $j\in I_k$ ($k=-1,0,1$),}\\
0\quad &\textrm{otherwise.}
\end{array}
\right.
\end{equation*}
Since $(J,\la_{\mm},\chi_{\mm})$ satisfies the PEC for $\widehat{\WW}(A^{(1)}_J)$ by Lemma~\ref{lemma:translation_invariance1}, the triple $(J,\la^{0}-\mm,\chi^{0}+\mm)$ satisfies the PEC for $\WW(A_J)$. In particular, since $m$ is not an extremal value of the set $\{\la_j-m_j \ | \ j\in J\}$ and since $\la_j-m_j=m$ for all $j\in I_0$, we deduce from \cite[Lemma~4.4]{PEClocfin} that $D(I_0)$ is bounded. Repeating this argument with $I_0$ and $I_{\pm 1}$ permuted yields that each of the three sets $D(I_k)$, $k=-1,0,1$, is bounded, and hence that $D(J_m)$ is bounded, as desired. This proves (1).

(2) Since $D(J)$ is bounded, the minimal and maximal accumulation points $r_{\min}$ and $r_{\max}$ of $J$ indeed exist. Moreover, since $\Lambda(J)$ is finite, there exist some $m,n\in\RR$ such that $r_{\max}$ is an accumulation point for $J_m$ and $r_{\min}$ is an accumulation point for $J_n$. In particular, $m,n\in\Lambda(J)$, so that 
$$-\tfrac{1}{2}\leq m,n<\tfrac{1}{2}.$$
For short, we set $r_m:=r_{\max}$ and $r_n:=r_{\min}$.
Fix some $\epsilon>0$, and choose some (disjoint) infinite countable subsets $I_m=\{i_1,i_2,\dots\}\subseteq J_m$ and $I_n=\{j_1,j_2,\dots\}\subseteq J_n$ such that 
$$\sum_{s\in\NN}{|r_m-d_{i_s}|}<\epsilon\quad\textrm{and}\quad\sum_{s\in\NN}{|r_n-d_{j_s}|}<\epsilon.$$
For each $k\in\NN$, let $w_k\in S_{(J)}$ be the product of the transpositions $\tau_s$, $s=1,\dots,k$, exchanging $i_s$ and $j_s$. We also set
$$x_k=\sum_{s=1}^{k}{(e_{i_s}-e_{j_s})}\in\TTT(A^{(1)}_J).$$ We deduce from (\ref{eqn:second_eqn}) that for each $k\in\NN$,
\begin{equation*}
\begin{aligned}
\la(\tau_{x_k}w_k\inv.\chi-\chi)&=k+\sum_{s=1}^{k}{(d_{j_s}-d_{i_s})}+\sum_{s=1}^{k}{(\la_{i_s}-\la_{j_s})}+\sum_{s=1}^{k}{(\la_{i_s}-\la_{j_s})(d_{j_s}-d_{i_s})}\\
&=\sum_{s=1}^{k}{(1+\la_{i_s}-\la_{j_s})(1+d_{j_s}-d_{i_s})}=\sum_{s=1}^{k}{(1+m-n)(1+d_{j_s}-d_{i_s})}\\
&= k(1+m-n)(1+r_n-r_m)+(1+m-n)\sum_{s=1}^{k}{\Big((d_{j_s}-r_n)-(d_{i_s}-r_m)\Big)}\\
&\leq k(1+m-n)(1+r_n-r_m) + 2\epsilon (1+m-n).
\end{aligned}
\end{equation*}
Since $k\in\NN$ was arbitrary, the PEC for $\widehat{\WW}(A^{(1)}_J)$ then implies that
$$(1+m-n)(1+r_n-r_m)\geq 0.$$
Since $1+m-n>0$ by assumption, we deduce that $r_m-r_n\leq 1$, yielding (2).

(3) Finally, assume that $r_{\max}-r_{\min}= 1$. We prove that $\sum_{j\in J_+}{(d_j-r_{\max})}<\infty$, the proof for $J_-$ being similar. If $J_+$ is finite, there is nothing to prove. Assume now that $J_+$ is infinite. Let $I_+=\{i_1,i_2,\dots\}$ be an arbitrary infinite countable subset of $J_+$, so that 
$$d_{i_s}-r_{\max}>0\quad\textrm{for all $s\in\NN$.}$$
Let $\epsilon>0$ be such that $$|\la_i-\la_j|<1-\epsilon\quad\textrm{for all $i,j\in J$,}$$ and choose some infinite countable subset $I_-=\{j_1,j_2,\dots\}\subseteq J\setminus I_+$ such that 
$$\sum_{s\in\NN}{|r_{\min}-d_{j_s}|}<\epsilon.$$
Defining the elements $x_k$, $w_k$ for $k \in \NN$ 
as above, we deduce from the equality $r_{\max}-r_{\min}=1$ that
\begin{equation*}
\begin{aligned}
\la(\tau_{x_k}w_k\inv.\chi-\chi)&=\sum_{s=1}^{k}{(1+\la_{i_s}-\la_{j_s})(1+d_{j_s}-d_{i_s})}\\
&=\sum_{s=1}^{k}{(1+\la_{i_s}-\la_{j_s})(d_{j_s}-r_{\min})}+\sum_{s=1}^{k}{(1+\la_{i_s}-\la_{j_s})(r_{\max}-d_{i_s})}\\
&\leq 2\sum_{s=1}^{k}{|r_{\min}-d_{j_s}|}-\sum_{s=1}^{k}{(1+\la_{i_s}-\la_{j_s})\cdot|d_{i_s}-r_{\max}|}\\
&\leq 2\epsilon - \epsilon \sum_{s=1}^{k}{|d_{i_s}-r_{\max}|}.
\end{aligned}
\end{equation*}
Since $k\in\NN$ was arbitrary, we deduce that
$$\sum_{j\in I_+}{(d_{j}-r_{\max})}<2 - \epsilon^{-1}
\inf \la\big(\widehat{\WW}(A^{(1)}_J).\chi-\chi\big)$$
for any infinite countable subset $I_+$ of $J_+$, yielding the corresponding assertion for $I_+$ replaced by $J_+$. This concludes the proof of (3).
\end{proof}

For a tuple $\nu=(\nu_j)_{j\in J}\in\RR^J$, we set
$$|\nu|:=(|\nu_j|)_{j\in J}\in\RR^J.$$

\begin{prop}\label{prop:necessary_B_aff}
Assume that $\la^{0}=\fr{\la^{0}}$ and that $(J,\la,\chi)$ satisfies the PEC for $\widehat{\WW}(C^{(1)}_J)$. Let $m\in\RR$. Then the following hold:
\begin{enumerate}
\item
If $|m|<1/2$, then $J_m^{>1/2}$ and $J_m^{<-1/2}$ are summable.
\item
If $|m|=1/2$, then $J_m^{>1}$ and $J_m^{<-1}$ are summable.
\item
$(J,-|\la^{0}|,|\fr{\chi^{0}}|)$ satisfies the PEC for $\WW(A_J)$.
\item
$\sum_{j\in J_+}{|\la_jd_j|}<\infty$, where $J_+:=\{j\in J \ | \ \la_jd_j>0\}$.
\end{enumerate}
\end{prop}

\begin{proof}
(1) Assume that $|m|<1/2$. Fix some $\epsilon>0$ such that $|m|<1/2-\epsilon$. 
Let $I$ be any finite subset of $J_m^{>1/2}$ (resp. $J_m^{<-1/2}$), so that $|d_j|>1/2$ for all $j\in I$. Let $\sigma=(\sigma_j)_{j\in J}\in\{\pm 1\}^{(J)}$ with support in $I$ be such that $m\sigma_j\fr{d_j}\leq 0$ for all $j\in I$. Then, for all $i\in I$,
$$\bfr{m+\sigma_jd_j}^2=\bfr{m+\sigma_j\fr{d_j}}^2=\bfr{|m|-|\fr{d_j}|}^2=\big(|m|-|\fr{d_j}|\big)^2.$$
Consider also the element $x=\sum_{j\in I}{n_je_j}\in\TTT(C^{(1)}_J)$ defined by $n_j=-[m+\sigma_jd_j]$ for all $j\in I$. We then deduce from (\ref{eqn:second_eqn}) that
\begin{equation*}
\begin{aligned}
\la(\tau_x\sigma.\chi-\chi)&=\frac{1}{2}\sum_{j\in I}{\Big((n_j+m+\sigma_jd_{j})^2-(m+d_j)^2\Big)}\\
&=\frac{1}{2}\sum_{j\in I}{\Big(\bfr{m+\sigma_jd_j}^2-(m+d_j)^2\Big)}\\
&\leq\frac{1}{2}\sum_{j\in I}{\Big((|m|-|\fr{d_j}|)^2-(|m|-|d_j|)^2\Big)}\\
&=-\frac{1}{2}\sum_{j\in I}{\big(|d_j|-|\fr{d_j}|\big)\big(|d_j|+|\fr{d_j}|-2|m|\big)}.\\
\end{aligned}
\end{equation*}
Note that $|d_j|+|\fr{d_j}|\geq 1$ for all $j\in I$, and hence $|d_j|+|\fr{d_j}|-2|m|>2\epsilon$ for all $j\in I$. Since moreover 
$$|d_j|-|\fr{d_j}|=(|d_j|-1/2)+(1/2-|\fr{d_j}|)> 0\quad\textrm{for all $j\in I$},$$
we deduce that 
\begin{equation}\label{eqn:PECB1}
\la(\tau_x\sigma.\chi-\chi) \leq -\epsilon\sum_{j\in I}{\big(|d_j|-|\fr{d_j}|\big)}
\leq -\epsilon\sum_{j\in I}{(|d_j|-1/2)}.
\end{equation}
Since the finite subset $I$ of $J_m^{>1/2}$ (resp. $J_m^{<-1/2}$) was arbitrary, the PEC for $\widehat{\WW}(C^{(1)}_J)$ implies that $J_m^{>1/2}$ (resp. $J_m^{<-1/2}$) is summable, proving (1). 

(2) Assume next that $|m|=1/2$. Let $I$ be any finite subset of $J_m^{>1}$ (resp. $J_m^{<-1}$), so that $|d_j|>1$ for all $j\in I$. Defining $\sigma$ and $x$ as above, we again get that
\begin{equation*}
\begin{aligned}
\la(\tau_x\sigma.\chi-\chi)\leq -\frac{1}{2}\sum_{j\in I}{\big(|d_j|-|\fr{d_j}|\big)\big(|d_j|+|\fr{d_j}|-2|m|\big)}.\\
\end{aligned}
\end{equation*}
Since $|d_j|-|\fr{d_j}|>1/2$ and $|d_j|+|\fr{d_j}|-2|m|=|d_j|+|\fr{d_j}|-1>0$ for all $j\in I$, we deduce that
\begin{equation}\label{eqn:PECB2}
\la(\tau_x\sigma.\chi-\chi) \leq -\frac{1}{4}\sum_{j\in I}{\big(|d_j|+|\fr{d_j}|-1\big)}\leq -\frac{1}{4}\sum_{j\in I}{(|d_j|-1)}.
\end{equation}
Since the finite subset $I$ of $J_m^{>1}$ (resp. $J_m^{<-1}$) was arbitrary, the PEC for $\widehat{\WW}(C^{(1)}_J)$ implies that $J_m^{>1}$ (resp. $J_m^{<-1}$) is summable, proving (2).

(3) 
We now prove that $(J,-|\la^{0}|,|\fr{\chi^{0}}|)$ satisfies the PEC for $\WW(A_J)$. Let $w\in S_{(J)}$, and let $I$ be a finite subset of $J$ containing the support of $w$. Let $\sigma=(\sigma_j)_{j\in J}\in\{\pm 1\}^{(J)}$ with support in $I$ be such that $\la_j\sigma_j\fr{d_{w(j)}}\leq 0$ for all $j\in I$. Consider also the element $x=\sum_{j\in I}{n_je_j}\in\TTT(C^{(1)}_J)$ defined by $n_j=-[\la_j+\sigma_jd_{w(j)}]$ for all $j\in I$. As before, we deduce from (\ref{eqn:second_eqn}) that
\begin{equation*}
\begin{aligned}
\la(\tau_x\sigma w\inv.\chi-\chi)&=\frac{1}{2}\sum_{j\in I}{\Big((n_j+\la_j+\sigma_jd_{w(j)})^2-(\la_j+d_j)^2\Big)}\\
&=\frac{1}{2}\sum_{j\in I}{\Big(\bfr{\la_j+\sigma_jd_{w(j)}}^2-(\la_j+d_j)^2\Big)}\\
&\leq\frac{1}{2}\sum_{j\in I}{\Big((|\la_j|-|\fr{d_{w(j)}}|)^2-(|\la_j|-|d_j|)^2\Big)}\\
&=\sum_{j\in I}{|\la_j|\cdot (|\fr{d_j}|-|\fr{d_{w(j)}}|)}+\frac{1}{2}\sum_{j\in I}{\big(|\fr{d_j}|-|d_j|\big)\big(|d_j|+|\fr{d_j}|-2|\la_j|\big)}.
\end{aligned}
\end{equation*}
Since $|\fr{d_j}|-|d_j|=0$ for $|d_j|\leq 1/2$, while $|d_j|+|\fr{d_j}|-2|\la_j|\geq 0$ for $|d_j|\geq 1/2$, we deduce that 
$$\big(|\fr{d_j}|-|d_j|\big)\big(|d_j|+|\fr{d_j}|-2|\la_j|\big)\leq 0\quad\textrm{for all $j\in I$,}$$
and hence
\begin{equation}\label{eqn:PECB3}
\la(\tau_x\sigma w\inv.\chi-\chi)\leq \sum_{j\in I}{|\la_j|\cdot (|\fr{d_j}|-|\fr{d_{w(j)}}|)}=-|\la^{0}|\big(w\inv.|\fr{\chi^{0}}|-|\fr{\chi^{0}}|\big).
\end{equation}
As $w\in S_{(J)}=\WW(A_J)$ was arbitrary, this proves that $(J,-|\la^{0}|,|\fr{\chi^{0}}|)$ satisfies the PEC for $\WW(A_J)$, yielding (3).

(4) Finally, let $I$ be an arbitrary finite subset of $J_+$, and consider the element $\sigma=(\sigma_j)_{j\in J}\in\{\pm 1\}^{(J)}$ with support $I$. Then $$\la(\sigma.\chi-\chi)=\sum_{j\in J}{\la_j(\sigma_jd_j-d_j)}=-2\sum_{j\in I}{\la_jd_j}=-2\sum_{j\in I}{|\la_jd_j|}.$$
Hence $\sum_{j\in I}{|\la_jd_j|}\leq -\tfrac{1}{2}\cdot\inf\la(\WW(C_J).\chi-\chi)$ for any finite subset $I\subseteq J_+$, proving (4). This concludes the proof of the proposition.
\end{proof}

\section{Minimal energy sets}\label{Section:MES}
We first characterise the triples $(J,\la,\chi)$ of minimal energy for $\widehat{\WW}$, where as before $\la=(1,\la^0,0)$ and $\chi=(0,\chi^0,1)$ for some tuples $\la^0,\chi^0\in\RR^J$. Although, for the purpose of proving Theorem~\ref{thmintro:A}, we will only need to describe such triples for the types $A_J^{(1)}$ and $C^{(1)}_J$, we also provide, for the sake of completeness, explicit descriptions for the other types.

\begin{definition}\label{definition:min_sets_affine}
For a tuple $\la^{0}=(\la_j)_{j\in J}\in\RR^J$ with $\la^{0}=\fr{\la^{0}}$, we define the following convex subsets of $\RR^J$. 
We set
$$C_{\min}(\la^{0},A^{(1)}_J)=\RR\cdot \mathbf{1}+\big\{(d_j)_{j\in J}\in\RR^J \ | \ \forall j\in J: \ |d_j|\leq \tfrac{1}{2}\quad\textrm{and}\quad \forall i,j\in J: \ \la_i<\la_j\implies d_i\geq d_j\big\},$$
where $\mathbf{1}\in\RR^J$ is the constant function~$1$.
We also let $C_{\min}(\la^{0},C^{(1)}_J)$ denote the set of $(d_j)_{j\in J}\in\RR^J$ satisfying the following four conditions:
\begin{enumerate}
\item[(C1)]
$\forall j\in J: \ |\la_j|<\tfrac{1}{2}\implies |d_j|\leq \tfrac{1}{2}$.
\item[(C2)]
$\forall j\in J: \ |\la_j|=\tfrac{1}{2}\implies |d_j|\leq 1$.
\item[(C3)]
$\forall j\in J: \ \la_jd_j\leq 0$.
\item[(C4)]
$\forall i,j\in J: \ |\la_i|<|\la_j|\implies |\fr{d_i}|\leq |\fr{d_j}|$.
\end{enumerate}
\end{definition}

The following lemma, which contains some observations 
that will be used in the sequel, is an easy exercise.  
\begin{lemma}\label{lemma:fraction_part}
Let $a,b,x\in\RR$.
\begin{enumerate}
\item If $-\tfrac{1}{2}\leq a,b\leq\tfrac{1}{2}$, then $\fr{a\pm b}^2\geq (|a|-|b|)^2$.
\item If $|x| \leq \tfrac{1}{2}$, then $|x| = |\fr x |$. 
\item If $\frac{1}{2} \leq |x| \leq 1$, then $|x| + |\fr x |=1$. 
\item If $|x| \geq 1$, then $|x| - |\fr x |\geq 1$.
\end{enumerate}
\end{lemma}

\begin{prop}\label{prop:minset_B_aff}
Assume that $\la^{0}=\fr{\la^{0}}$. Then the following are equivalent:
\begin{enumerate}
\item
$\inf\la\big(\widehat{\WW}(C^{(1)}_J).\chi-\chi\big)=0$.
\item
$\chi^{0}\in C_{\min}(\la^{0},C^{(1)}_J)$.
\end{enumerate}
\end{prop}

\begin{proof}
(1)$\Rightarrow$(2):
Assume first that $\inf\la\big(\widehat{\WW}(C^{(1)}_J).\chi-\chi\big)=0$. We have to prove that $\chi^{0}=(d_j)_{j\in J}$ satisfies the four conditions (C1)--(C4) from Definition~\ref{definition:min_sets_affine}. 
Since $\WW(B_J)\subseteq \widehat{\WW}(C^{(1)}_J)$, it follows from Proposition~\ref{prop:minimality} that $\chi^{0}\in C_{\min}(\la^{0},B_J)$, so that (C3) is satisfied. To check (C1), let $j\in J$ be such that $|\la_j|<1/2$ and assume for a contradiction that $|d_j|>1/2$. Fix some $\epsilon>0$ such that $|\la_j|<1/2-\epsilon$. It then follows from (\ref{eqn:PECB1}) that 
$$\la(\tau_x\sigma.\chi-\chi) \leq -\epsilon(|d_j|-1/2)<0$$
for some suitable $\sigma\in\{\pm 1\}^{(J)}$ and $x\in \TTT(C^{(1)}_J)$ with support in $I:=\{j\}$, contradicting (1). Similarly, to check (C2), let $j\in J$ be such that $|\la_j|=1/2$ and assume for a contradiction that $|d_j|>1$. It then follows from (\ref{eqn:PECB2}) that 
$$\la(\tau_x\sigma.\chi-\chi) \leq -\frac{1}{4}(|d_j|-1)<0$$
for some suitable $\sigma\in\{\pm 1\}^{(J)}$ and $x\in \TTT(C^{(1)}_J)$ with support in $I:=\{j\}$, again contradicting (1). Finally, to check (C4), let $i,j\in J$ with $i\neq j$ and let $w\in S_{(J)}$ be the transposition exchanging $i$ and $j$. It then follows from (\ref{eqn:PECB3}) that
$$\la(\tau_x\sigma w\inv.\chi-\chi)\leq |\la_i|\cdot (|\fr{d_i}|-|\fr{d_{w(i)}}|)+|\la_j|\cdot (|\fr{d_j}|-|\fr{d_{w(j)}}|)=(|\la_i|-|\la_j|)(|\fr{d_i}|-|\fr{d_{j}}|)$$
for some suitable $\sigma\in\{\pm 1\}^{(J)}$ and $x\in \TTT(C^{(1)}_J)$ with support in $I:=\{i,j\}$, yielding the claim.

(2)$\Rightarrow$(1):
Assume next that $\chi^{0}\in C_{\min}(\la^{0},C^{(1)}_J)$. 
Let $\widehat{w}=\tau_x\sigma w\inv\in\widehat{\WW}(C^{(1)}_J)$, for some $w\in S_{(J)}$, some $\sigma=(\sigma_j)_{j\in J}\in\{\pm 1\}^{(J)}$ and some $x=\sum_{j\in J}{n_je_j}\in\TTT(C^{(1)}_J)$. Let $I$ be some finite subset of $J$ containing the supports of $x$, $\sigma$ and $w$. 
It follows from (\ref{eqn:second_eqn}) that
\begin{equation*}
\begin{aligned}
\la(\widehat{w} .\chi-\chi) & =\frac{1}{2}\sum_{j\in I}{\Big((n_j+\la_j+\sigma_jd_{w(j)})^2-(\la_j+d_j)^2\Big)}\\
& \geq \frac{1}{2}\sum_{j\in I}{\Big(\fr{\la_j+\sigma_jd_{w(j)}}^2-(\la_j+d_j)^2\Big)}\\
& = \frac{1}{2}\sum_{j\in I}{\Big(\fr{\la_j+\sigma_j\fr{d_{w(j)}}}^2-(\la_j+d_j)^2\Big)}\\
& \geq \frac{1}{2}\sum_{j\in I}{\Big(\big(\big|\la_j\big|-\big|\fr{d_{w(j)}}\big|\big)^2-(\la_j+d_j)^2\Big)}\qquad \qquad \mbox{(Lemma~\ref{lemma:fraction_part}(1))} \\
&= \frac{1}{2}\sum_{j\in I}{\Big(\big(\big|\la_j\big|-\big|\fr{d_{w(j)}}\big|\big)^2-(|\la_j|-|d_j|)^2\Big)}\qquad \qquad \mbox{by (C3).} 
\end{aligned}
\end{equation*}
On the other hand, the rearrangement inequality (see \cite[Lemma~2.2]{PEClocfin}) and (C4) imply that 
$$\sum_{j\in I}{|\la_j|\cdot\big(-\big|\fr{d_{w(j)}}\big|+\big|\fr{d_j}\big|\big)}\geq 0.$$
Hence
\begin{equation*}
\begin{aligned}
\la(\widehat{w} .\chi-\chi) & \geq \frac{1}{2}\sum_{j\in I}{\Big(\big(\big|\la_j\big|-\big|\fr{d_{w(j)}}\big|\big)^2-(|\la_j|-|d_j|)^2\Big)}\\
&= \frac{1}{2}\sum_{j\in I}{\big(\big|\fr{d_{j}}\big|^2-\big|d_{j}\big|^2\big)}+\sum_{j\in I}{|\la_j|\cdot\big(-\big|\fr{d_{w(j)}}\big|+\big|d_j\big|\big)}\\
&\geq \frac{1}{2}\sum_{j\in I}{\big(\big|\fr{d_{j}}\big|^2-\big|d_{j}\big|^2\big)}+\sum_{j\in I}{|\la_j|\cdot\big(\big|d_{j}\big|-\big|\fr{d_{j}}\big|\big)}\\
&=\frac{1}{2}\sum_{j\in I}{\big(\big|\fr{d_{j}}\big|-\big|d_{j}\big|\big)\big(\big|\fr{d_{j}}\big|+\big|d_{j}\big|-2|\la_j|\big)}.
\end{aligned}
\end{equation*}
Let now $j\in I$. If $|\la_j|<1/2$, then $|d_j|\leq 1/2$ by (C1). On the other hand, if $|\la_j|=1/2$, then $|d_j|\leq 1$ by (C2). Hence either $|d_j|\leq 1/2$, in which case $\big|\fr{d_{j}}\big|-\big|d_{j}\big|=0$, or else $1/2<|d_j|\leq 1$ and $|\la_j|=1/2$, in which case 
$$\big|\fr{d_{j}}\big|+\big|d_{j}\big|-2|\la_j|=\big|\fr{d_{j}}\big|+\big|d_{j}\big|-1=0.$$
Thus $\big(\big|\fr{d_{j}}\big|-\big|d_{j}\big|\big)\big(\big|\fr{d_{j}}\big|+\big|d_{j}\big|-2|\la_j|\big)=0$ for all $j\in I$, so that 
$$\la(\widehat{w} .\chi-\chi)\geq 0.$$
Since $\widehat{w}\in\widehat{\WW}(C^{(1)}_J)$ was arbitrary, this concludes the proof of the proposition.
\end{proof}

To describe the triples $(J,\la,\chi)$ of minimal energy for $\widehat{\WW}(A^{(1)}_J)$, we could proceed as in the proof of Proposition~\ref{prop:minset_B_aff}. However, it will be easier to start from the description provided by \cite{convexhull}, which we now recall.

By \cite[Theorem~3.5(i)]{convexhull}, the set $\widehat{\WW}.\la-\la$ is contained in the cone $-C_{\la}$, where $$C_{\la}:=\mathrm{cone}\{\gamma\in\widehat{\Delta} \ | \ \la(\check{\gamma})>0\}.$$ Conversely, if $\gamma\in \widehat{\Delta}\cap C_{\la}$, then $-\gamma\in \RR_+(r_{\gamma}(\la)-\la)\subseteq \mathrm{cone}\{\widehat{\WW}.\la-\la\}$ because $r_{\gamma}(\la)=\la-\la(\check{\gamma})\gamma$. This shows that $\mathrm{cone}\{\widehat{\WW}.\la-\la\}=-C_{\la}$. In particular, the triple $(J,\la,\chi)$ is of minimal energy for $\widehat{\WW}$ if and only if 
\begin{equation}\label{eqn:mincond}
\la(\check{\gamma})>0\implies \gamma(\chi)\leq 0 \quad\textrm{for all $\gamma\in\widehat{\Delta}$.}
\end{equation}

\begin{prop}\label{prop:minset_A_aff}
Assume that $\la^{0}=\fr{\la^{0}}$. Then the following are equivalent:
\begin{enumerate}
\item
$\inf\la\big(\widehat{\WW}(A^{(1)}_J).\chi-\chi\big)=0$.
\item
$\chi^{0}\in C_{\min}(\la^{0},A^{(1)}_J)$.
\end{enumerate}
\end{prop}
\begin{proof}
Set $\widehat{\WW}=\widehat{\WW}(A^{(1)}_J)$ and let $\widehat{\Delta}$ be the corresponding root system of type $A^{(1)}_J$. Note that the coroot of $(0,\epsilon_j-\epsilon_k,n)\in\widehat{\Delta}$ is given by $(-n,e_j-e_k,0)$. By (\ref{eqn:mincond}), the first statement of the proposition is thus equivalent to the condition
\begin{equation}\label{eqn:condA}
\la_j-\la_k>n\implies d_j-d_k\leq -n \quad\textrm{for all distinct $j,k\in J$ and $n\in\ZZ$.}
\end{equation}
Since $-\tfrac{1}{2}\leq \la_j<\tfrac{1}{2}$ for all $j\in J$ by assumption, 
condition (\ref{eqn:condA}) is empty for $n\geq 1$. For $n=0$, we get
\begin{equation}\label{eqn:condA1}
\la_j>\la_k\implies d_j\leq d_k \quad\textrm{for all distinct $j,k\in J$.}
\end{equation}
Similarly, for $n\leq -1$, we get
\begin{equation}\label{eqn:condA2}
d_j-d_k\leq 1 \quad\textrm{for all distinct $j,k\in J$.}
\end{equation}
In turn, the conditions (\ref{eqn:condA1}) and (\ref{eqn:condA2}) are satisfied if and only if $\chi^{0}\in C_{\min}(\la^{0},A^{(1)}_J)$, as desired.
\end{proof}

We now describe the minimal energy sets for the other types.

\begin{definition}
Let $\la^0\in\RR^J$ with $\la^0=\fr{\la^0}$. For $X\in\{B^{(1)}_J,D_J^{(1)},B_J^{(2)},C_J^{(2)},BC_J^{(2)}\}$, we denote by $C_{\min}(\la^{0},X)$ the set of tuples $\chi^0\in\RR^J$ such that $\inf\la\big(\widehat{\WW}(X).\chi-\chi\big)=0$, where $\la=(1,\la^0,0)$ and $\chi=(0,\chi^0,1)$. Note that, in view of Propositions~\ref{prop:minset_B_aff} and \ref{prop:minset_A_aff}, this definition is coherent with the corresponding notation for $X\in \{A_J^{(1)},C^{(1)}_J\}$.
\end{definition}

\begin{lemma}\label{lemma:minset_D_aff}
Assume that $\la^{0}=\fr{\la^{0}}$. Then $\chi^0\in C_{\min}(\la^{0},D_J^{(1)})$ if and only if $\chi^0\in C_{\min}(\la^{0},D_J)$ and the following conditions hold:
\begin{enumerate}
\item[(D1)] For all distinct $j,k\in J$ : $-1\leq d_j-d_k\leq 1$. 
\item[(D2)] For all distinct $j,k\in J$ : $-1\leq d_j+d_k\leq 2$.
\item[(D3)] For all distinct $j,k\in J$ : $d_j+d_k\leq 1$ or $\la_j=\la_k=-\tfrac{1}{2}$.
\end{enumerate}
\end{lemma}
\begin{proof}
The condition (\ref{eqn:mincond}) for the roots of the form $\gamma=(0,\pm(\epsilon_j+\epsilon_k),n)$ with $j\neq k$ can be rewritten as
\begin{equation}\label{eqn:condD}
\pm(\la_j+\la_k)>n\implies \pm(d_j+d_k)\leq -n \quad\textrm{for all distinct $j,k\in J$ and $n\in\ZZ$,}
\end{equation}
so that $C_{\min}(\la^{0},D_J^{(1)})$ is characterised by the conditions (\ref{eqn:condA}) and (\ref{eqn:condD}).
For $n=0$, the conditions (\ref{eqn:condA}) and (\ref{eqn:condD}) amount to $\chi^0\in C_{\min}(\la^{0},D_J)$.
Since $-\tfrac{1}{2}\leq \la_j<\tfrac{1}{2}$ for all $j\in J$ by assumption, the condition (\ref{eqn:condD}) is empty for $n\geq 1$, and equivalent to
(D2) and (D3) for $n\leq -1$. Finally, the condition (\ref{eqn:condA}) for $n\neq 0$ amounts to (D1). This concludes the proof of the lemma.
\end{proof}

For an explicit description of the set $C_{\min}(\la^{0},D_J)$, we refer to \cite[Remark~5.9]{PEClocfin}.

\begin{lemma}
Assume that $\la^{0}=\fr{\la^{0}}$. Then $\chi^0\in C_{\min}(\la^{0},B_J^{(1)})$ if and only if one of the following holds:
\begin{enumerate}
\item
$\la_j=-1/2$ and $0\leq d_j\leq 1$ for all $j\in J$.
\item
There is some $i\in J$ with $\la_i\neq -\tfrac{1}{2}$, 
$\chi^0\in C_{\min}(\la^{0},B_J)$,  and (D1)--(D3) hold.
\end{enumerate}
\end{lemma}
\begin{proof}
The condition (\ref{eqn:mincond}) for the roots of the form $\gamma=(0,\pm \epsilon_j,n)$ is equivalent to
\begin{equation}\label{eqn:condB}
\pm \la_j>n\implies \pm d_j\leq -n \quad\textrm{for all $j\in J$ and $n\in\ZZ$,}
\end{equation}
so that $C_{\min}(\la^{0},B_J^{(1)})$ is characterised by the conditions (\ref{eqn:condA}),  (\ref{eqn:condD}) and (\ref{eqn:condB}).
For $n=0$, these three conditions amount to $\chi^0\in C_{\min}(\la^{0},B_J)$. As we saw in the proof of Lemma~\ref{lemma:minset_D_aff}, the conditions (\ref{eqn:condA}) and (\ref{eqn:condD}) for $n\neq 0$ are equivalent to the conditions (D1)--(D3). Finally, since $-\tfrac{1}{2}\leq \la_j<\tfrac{1}{2}$ for all $j\in J$ by assumption, the condition (\ref{eqn:condB}) is empty for $n\geq 1$, and, for $n \leq -1$, it is equivalent to
\begin{equation}\label{eqn:lastcond}
|d_j|\leq 1\quad\textrm{for all $j\in J$}.
\end{equation}
Assume first that there is some $i\in J$ such that $\la_i\neq -\tfrac{1}{2}$. Then for any $j\in J$ the conditions (D1)--(D3) imply that $-1\leq d_j\pm d_i\leq 1$ and hence that $-2\leq 2d_j\leq 2$, so that (\ref{eqn:lastcond}) holds. Thus, in that case, $\chi^0\in C_{\min}(\la^{0},B_J^{(1)})$ if and only if (2) holds. Assume next that $\la_j=-1/2$ for all $j\in J$. Then $\chi^0\in C_{\min}(\la^{0},B_J)$ implies that $d_j\geq 0$ for all $j\in J$, while (\ref{eqn:lastcond}) implies that $d_j\leq 1$ for all $j\in J$. Conversely, if (1) holds, then it is easy to check that $\chi^0\in C_{\min}(\la^{0},B_J)$ and that the conditions (D1)--(D3) and (\ref{eqn:lastcond}) hold, as desired.
\end{proof}

\begin{lemma}
Assume that $\la^{0}=\fr{\la^{0}}$. Then $$C_{\min}(\la^{0},BC_J^{(2)})=C_{\min}(\la^{0},C^{(1)}_J)\quad\textrm{and}\quad C_{\min}(\la^{0},C^{(2)}_J)=C_{\min}(\la^{0},B_J^{(1)}).$$
\end{lemma}
\begin{proof}
This readily follows from the fact that $\widehat{\WW}(BC^{(2)}_J)=\widehat{\WW}(C^{(1)}_J)$ and $\widehat{\WW}(B^{(1)}_J)=\widehat{\WW}(C^{(2)}_J)$.
\end{proof}

\begin{lemma}\label{lemma:minset_B2_aff}
Let $\la=(1,\la^0,0)$ and $\chi=(0,\chi^0,1)$, and set $\la_2:=(1,\tfrac{\la^0}{2},0)$ and $\chi_2:=(0,\tfrac{\chi^0}{2},1)$. Then 
$$\inf \la\big(\widehat{\WW}(B^{(2)}_J).\chi-\chi\big)=4\cdot\inf\la_2\big(\widehat{\WW}(C^{(1)}_J).\chi_2-\chi_2\big).$$
In particular, if $\la^{0}=\fr{\la^{0}}$, then $C_{\min}(\la^{0},B^{(2)}_J)=2\cdot C_{\min}(\tfrac{\la^{0}}{2},C^{(1)}_J)$.
\end{lemma}
\begin{proof}
Let $\WW=S_{(J)}\ltimes \{\pm 1\}^{(J)}$, so that $\widehat{\WW}(C^{(1)}_J)=\WW\ltimes \tau(Q)$ and $\widehat{\WW}(B^{(2)}_J)=\WW\ltimes \tau(2Q)$. Then for any $x=\sum_{j\in J}{n_je_j}\in Q$, any $w\in S_{(J)}$ and any $\sigma=(\sigma_j)_{j\in J}\in \{\pm 1\}^{(J)}$, we get from (\ref{eqn:second_eqn}) that
\begin{equation*}
\begin{aligned}
\la(\tau_{2x}\sigma w\inv.\chi-\chi)
&=\frac{1}{2}\sum_{j\in J}{\Big((2n_j+\la_j+\sigma_jd_{w(j)})^2-(\la_j+d_j)^2\Big)}\\
&=2\sum_{j\in J}{\bigg(\Big(n_j+\frac{\la_j}{2}+\sigma_j\frac{d_{w(j)}}{2}\Big)^2-\Big(\frac{\la_j}{2}+\frac{d_j}{2}\Big)^2\bigg)}\\
&=4\cdot\la_2(\tau_{x}\sigma w\inv.\chi_2-\chi_2),
\end{aligned}
\end{equation*}
yielding the claim.
\end{proof}

Using Lemmas~\ref{lemma:SC0}, \ref{lemma:simpleobs2} and \ref{lemma:translation_invariance1}, we can now restate the results of this section for general $\la=(\la_c,\la^{0},\la_d)$ and $\chi=(\chi_c,\chi^{0},\chi_d)$ with $\la_c\chi_d\neq 0$.

\begin{definition} \label{def:5.10}
Let $\la=(\la_c,\la^0,\la_d)\in\RR\times\RR^J\times\RR$ with $\la_c\neq 0$. For $X=X_J^{(1)}$ or $X_J^{(2)}$ one of the seven standard types from \S\ref{section:preliminaries:affine}, we set
$$C_{\min}(\la,X):=\Big\{\chi=(\chi_c,\chi^0,\chi_d)\in\RR\times\RR^J\times\RR \ \big| \ \la_c\chi_d>0\quad\textrm{and}\quad \tfrac{\chi^0}{\chi_d}+\big[\tfrac{\la^0}{\la_c}\big]\in C_{\min}(\bfr{\tfrac{\la^0}{\la_c}},X)\Big\}.$$
\end{definition}

\begin{theorem}\label{thm:MINgen}
Let $\la=(\la_c,\la^0,\la_d)\in\RR\times\RR^J\times\RR$ with $\la_c\neq 0$, and let $X=X_J^{(1)}$ or $X_J^{(2)}$ be one of the seven standard types from \S\ref{section:preliminaries:affine}. Then for any $\chi=(\chi_c,\chi^0,\chi_d)\in\RR\times\RR^J\times\RR$ with $\chi_d\neq 0$, the following assertions are equivalent:
\begin{enumerate} 
\item
$\inf \la\big(\widehat{\WW}(X).\chi-\chi\big)=0$.
\item
$\chi\in C_{\min}(\la,X)$.
\end{enumerate}
\end{theorem}

\section{Characterisation of the PEC for \texorpdfstring{$\la$ $\ZZ$-discrete}{lambda Z-discrete}}\label{Section:COTPFLZD}
%\section{Characterisation of the PEC for $\la$ $\ZZ$-discrete}\label{Section:COTPFLZD}
We are now in a position to prove an analogue of Theorem~\ref{theorem:charact_PEC_locally_finite} for triples $(J,\la,\chi)$.

\begin{lemma}\label{lemma:invariance_summable_aff2}
Let $X\in\big\{A^{(1)}_J,C^{(1)}_J\big\}$. Let $\la^{0},\chi^{0}\in\RR^J$ and 
$(\chi^{0})'\in \ell^1(J)$. Set $\la=(1,\la^{0},0)$, $\chi=(0,\chi^{0},1)$ and $\chi'=(0,(\chi^{0})',0)$. Then $(J,\la,\chi)$ satisfies the PEC for $\widehat{\WW}(X)$ if and only if $(J,\la,\chi+\chi')$ satisfies the PEC for $\widehat{\WW}(X)$.
\end{lemma}
\begin{proof}
Set $\widehat{\WW}=\widehat{\WW}(X)$, and assume that $(J,\la,\chi)$ satisfies the PEC for $\widehat{\WW}$. Since
$$\inf\la\big(\widehat{\WW}.(\chi+\chi')-(\chi+\chi')\big)\geq \inf\la\big(\widehat{\WW}.\chi-\chi)+\inf\la\big(\widehat{\WW}.\chi'-\chi'),$$
the triple $(J,\la,\chi+\chi')$ then satisfies the PEC for $\widehat{\WW}$ by Lemma~\ref{lemma:trivial_case00} and \cite[Lemma~5.4]{PEClocfin}. Replacing $\chi^{0}$ by $\chi^{0}+(\chi^{0})'$ and $(\chi^{0})'$ by $-(\chi^{0})'$, the converse follows.
\end{proof}

\begin{theorem}\label{thm:PECA}
Assume that $\la$ is $\ZZ$-discrete and that $-\tfrac{1}{2}\leq \la_j<\tfrac{1}{2}$ for all $j\in J$. Then the following are equivalent:
\begin{enumerate}
\item
$(J,\la,\chi)$ satisfies the PEC for $\widehat{\WW}(A^{(1)}_J)$.
\item
$\chi^{0}\in C_{\min}(\la^{0},A^{(1)}_J)+\ell^1(J)$.
\end{enumerate}
\end{theorem}
\begin{proof} (2) $\Rightarrow$ (1): 
If $\chi^{0}\in C_{\min}(\la^{0},A^{(1)}_J)+\ell^1(J)$, then $(J,\la,\chi)$ satisfies the PEC for $\widehat{\WW}(A^{(1)}_J)$ by Proposition~\ref{prop:minset_A_aff} and Lemma~\ref{lemma:invariance_summable_aff2}. 

(1) $\Rightarrow$ (2): 
Assume now that $(J,\la,\chi)$ satisfies the PEC for $\widehat{\WW}(A^{(1)}_J)$, and let us prove that, up to substracting from $\chi^{0}$ some element of $\ell^1(J)$, one has $\chi^{0}\in C_{\min}(\la^{0},A^{(1)}_J)$. Note that, by Lemma~\ref{lemma:invariance_summable_aff2}, replacing $\chi^{0}$ by $\chi^{0}-(\chi^{0})'$ for some $(\chi^{0})'\in \ell^1(J)$ does not affect the fact that $(J,\la,\chi)$ satisfies the PEC for $\widehat{\WW}(A^{(1)}_J)$.

Since $(J,\la^{0},\chi^{0})$ satisfies the PEC for $\WW(A_J)$, it follows from Theorem~\ref{theorem:charact_PEC_locally_finite} that, up to substracting from $\chi^{0}$ some element of $\ell^1(J)$, we may assume that $\chi^{0}\in C_{\min}(\la^{0},A_J)$. 
By Proposition~\ref{prop:necessary_A_aff}, the set $D(J)$ is bounded, and if $r_{\min}$ and $r_{\max}$ respectively denote the minimal and maximal accumulation points of $J$, then $r_{\max}-r_{\min}\leq 1$. Set
$$r:=\frac{r_{\min}+r_{\max}}{2}\in\RR.$$ Then, up to replacing $\chi^{0}$ by $\chi^{0}-r\cdot\mathbf{1}\in C_{\min}(\la^{0},A_J)$, we may moreover assume that $r=0$, so that
$$[r_{\min},r_{\max}]\subseteq [-\tfrac{1}{2},\tfrac{1}{2}].$$
Set $J_+:=\{j\in J \ | \ d_j>1/2\}$ and $J_-:=\{j\in J \ | \ d_j<-1/2\}$. Then either $|r_{\min}|,|r_{\max}|<\tfrac{1}{2}$, in which case $J_+$ and $J_-$ are both finite, or else $r_{\max}-r_{\min}= 1$, in which case 
\begin{equation*}
\sum_{j\in J_+}{(d_j-\tfrac{1}{2})}=\sum_{j\in J_+}{(d_j-r_{\max})}<\infty\quad\textrm{and}\quad\sum_{j\in J_-}{(-\tfrac{1}{2}-d_j)}=\sum_{j\in J_-}{(r_{\min}-d_j)}<\infty
\end{equation*}
by Proposition~\ref{prop:necessary_A_aff}(3). In all cases, the tuple $(\chi^{0})'=(d_j')_{j\in J}$ defined by 
\begin{equation*}
d_j'=
\left\{
\begin{array}{ll}
d_j-1/2\quad &\textrm{if $d_j>1/2$,}\\
d_j+1/2\quad &\textrm{if $d_j<-1/2$,}\\
0\quad &\textrm{otherwise}
\end{array}
\right.
\end{equation*}
belongs to $\ell^1(J)$. Note that $\chi^{0}-(\chi^{0})'\in C_{\min}(\la^{0},A_J)$. Indeed, this follows from the fact that for all $i,j\in J$:
$$d_i\leq d_j\implies d_i-d_i'\leq d_j-d_j'.$$
Hence, up to replacing $\chi^{0}$ by $\chi^{0}-(\chi^{0})'$, we may assume that $\chi^{0}\in [-\tfrac{1}{2},\tfrac{1}{2}]^J\cap C_{\min}(\la^{0},A_J)\subseteq C_{\min}(\la^{0},A^{(1)}_J)$. This concludes the proof of the theorem.
\end{proof}

\begin{theorem}\label{thm:PECB}
Assume that $\la$ is $\ZZ$-discrete and that $\la^{0}=\fr{\la^{0}}$. Then the following are equivalent:
\begin{enumerate}
\item
$(J,\la,\chi)$ satisfies the PEC for $\widehat{\WW}(C^{(1)}_J)$.
\item
$\chi^{0}\in C_{\min}(\la^{0},C^{(1)}_J)+\ell^1(J)$.
\end{enumerate}
\end{theorem}
\begin{proof}
(2) $\Rightarrow$ (1): 
If $\chi^{0}\in C_{\min}(\la^{0},C^{(1)}_J)+\ell^1(J)$, then $(J,\la,\chi)$ satisfies the PEC for $\widehat{\WW}(C^{(1)}_J)$ by Proposition~\ref{prop:minset_B_aff} and Lemma~\ref{lemma:invariance_summable_aff2}. 

(1) $\Rightarrow$ (2): 
Assume now that $(J,\la,\chi)$ satisfies the PEC for $\widehat{\WW}(C^{(1)}_J)$, and let us prove that, up to substracting from $\chi^{0}$ some element of $\ell^1(J)$, the four conditions (C1)--(C4) from Definition~\ref{definition:min_sets_affine} are satisfied by $\chi^{0}$. Note that, by Lemma~\ref{lemma:invariance_summable_aff2}, replacing $\chi^{0}$ by $\chi^{0}-(\chi^{0})'$ for some $(\chi^{0})'\in \ell^1(J)$ does not affect the fact that $(J,\la,\chi)$ satisfies the PEC for $\widehat{\WW}(C^{(1)}_J)$.

Since $(J,-|\la^{0}|,|\fr{\chi^{0}}|)$ satisfies the PEC for $\WW(A_J)$ by Proposition~\ref{prop:necessary_B_aff}(3), it follows from Theorem~\ref{theorem:charact_PEC_locally_finite} that, up to substracting from $\chi^{0}$ some element of $\ell^1(J)$, we may assume that $|\fr{\chi^{0}}|\in C_{\min}(-|\la^{0}|,A_J)$. In other words, we may assume that 
$$\forall i,j\in J: \ |\la_i|<|\la_j|\implies |\fr{d_i}|\leq |\fr{d_j}|,$$
hence that (C4) is satisfied.

We next claim that the tuple
$(\chi^{0})'=(d_j')_{j\in J}$ defined by 
\begin{equation*}
d_j'=
\left\{
\begin{array}{ll}
d_j-|\fr{d_j}|\quad &\textrm{if $|\la_j|<1/2$ and $d_j>1/2$,}\\
d_j+|\fr{d_j}|\quad &\textrm{if $|\la_j|<1/2$ and $d_j<-1/2$,}\\
0\quad &\textrm{otherwise}
\end{array}
\right.
\end{equation*}
belongs to $\ell^1(J)$. Indeed, this follows from Proposition~\ref{prop:necessary_B_aff}(1) and the fact that for $d_j>1/2$, one has $$\big|d_j-|\fr{d_j}|\big|=d_j-|\fr{d_j}|=(d_j-\tfrac{1}{2})+(\tfrac{1}{2}-|\fr{d_j}|)\leq 2(d_j-\tfrac{1}{2}),$$ while for $d_j<-1/2$, one has $$\big|d_j+|\fr{d_j}|\big|=-d_j-|\fr{d_j}|=(-d_j-\tfrac{1}{2})+(\tfrac{1}{2}-|\fr{d_j}|)\leq 2(-d_j-\tfrac{1}{2}).$$
Note moreover that 
$$|\fr{d_j-d'_j}|=|\fr{d_j}| \quad\textrm{for all $j\in J$},$$
so that $\chi^{0}-(\chi^{0})'$ still satisfies (C4). Hence, up to replacing $\chi^{0}$ by $\chi^{0}-(\chi^{0})'$, we may assume that $\chi^{0}$ satisfies (C1) and (C4).

Similarly, we claim that the tuple
$(\chi^{0})'=(d_j')_{j\in J}$ defined by 
\begin{equation*}
d_j'=
\left\{
\begin{array}{ll}
d_j+|\fr{d_j}|-1\quad &\textrm{if $|\la_j|=1/2$ and $d_j>1$,}\\
d_j-|\fr{d_j}|+1\quad &\textrm{if $|\la_j|=1/2$ and $d_j<-1$,}\\
0\quad &\textrm{otherwise}
\end{array}
\right.
\end{equation*}
belongs to $\ell^1(J)$. Indeed, this follows from Proposition~\ref{prop:necessary_B_aff}(2) and the fact that for $d_j>1$, one has $$\big|d_j+|\fr{d_j}|-1\big|=d_j-1+|\fr{d_j}|\leq 2(d_j-1),$$ while for $d_j<-1$, one has $$\big|d_j-|\fr{d_j}|+1\big|=-d_j-1+|\fr{d_j}|\leq 2(-d_j-1).$$
Since moreover
$$|\fr{d_j-d'_j}|=|\fr{d_j}| \quad\textrm{for all $j\in J$},$$
$\chi^{0}-(\chi^{0})'$ still satisfies (C4). Since clearly $\chi^{0}-(\chi^{0})'$ also still satisfies (C1), we may thus assume, up to replacing $\chi^{0}$ by $\chi^{0}-(\chi^{0})'$, that $\chi^{0}$ satisfies (C1), (C2) and (C4).

Finally, it follows from Proposition~\ref{prop:necessary_B_aff}(4) and the fact that $\Lambda(J)$ is finite that 
$$\sum_{j\in J_+}{|d_j|}<\infty,$$ where $J_+:=\{j\in J \ | \ \la_jd_j>0\}$.
Hence the tuple $(\chi^{0})'=(d_j')_{j\in J}$ defined by 
\begin{equation*}
d_j'=
\left\{
\begin{array}{ll}
2d_j\quad &\textrm{if $\la_jd_j>0$,}\\
0\quad &\textrm{otherwise}
\end{array}
\right.
\end{equation*}
belongs to $\ell^1(J)$. Again, as 
$$|\fr{d_j-d'_j}|=|\fr{d_j}| \quad\textrm{for all $j\in J$},$$
$\chi^{0}-(\chi^{0})'$ still satisfies (C4). Since clearly $\chi^{0}-(\chi^{0})'$ also still satisfies (C1) and (C2), we may then assume, up to replacing $\chi^{0}$ by $\chi^{0}-(\chi^{0})'$, that $\chi^{0}$ satisfies (C1), (C2), (C3) and (C4), and hence that $\chi^{0}\in C_{\min}(\la^{0},C^{(1)}_J)$. This concludes the proof of the theorem.
\end{proof}

\begin{prop}\label{prop:other_types}
Let $X$ be one of the types $B^{(1)}_J$, $D^{(1)}_J$, $C^{(2)}_J$, and $BC^{(2)}_J$. Assume that $\la$ is $\ZZ$-discrete. Then $(J,\la,\chi)$ satisfies the PEC for $\widehat{\WW}(X)$ if and only if it satisfies the PEC for $\widehat{\WW}(C^{(1)}_J)$.
\end{prop}
\begin{proof}
From \S\ref{subsection:Weyl_group}, we deduce the following inclusions:
\begin{equation}\label{eqn:inclusionsWhat}
\widehat{\WW}(A^{(1)}_J)\subseteq \widehat{\WW}(D^{(1)}_J)\subseteq \widehat{\WW}(B^{(1)}_J)= \widehat{\WW}(C^{(2)}_J)\subseteq \widehat{\WW}(BC^{(2)}_J)=\widehat{\WW}(C^{(1)}_J).\end{equation}
It is thus sufficient to prove that the PEC for $\widehat{\WW}(D^{(1)}_J)$ implies the PEC for $\widehat{\WW}(C^{(1)}_J)$. 

Let us thus assume that $(J,\la,\chi)$ satisfies the PEC for $\widehat{\WW}(D^{(1)}_J)$. In order to prove that it also satisfies the PEC for $\widehat{\WW}(C^{(1)}_J)$, we may assume by Lemma~\ref{lemma:translation_invariance1} that $\la^{0}=\fr{\la^{0}}$.
Moreover, the above inclusions show that $(J,\la,\chi)$ satisfies the PEC for $\widehat{\WW}(A^{(1)}_J)$. In particular, $D(J)$ is bounded by Proposition~\ref{prop:necessary_A_aff}(1). Set $C=\sup_{j\in J}{|d_j|}$. 

Let $\widehat{w}\in\widehat{\WW}(C^{(1)}_J)$, which we write as $\widehat{w}=\tau_x\sigma w\inv$ for some $x=\sum_{j\in J}{n_je_j}\in Q$, some $w\in S_{(J)}$ and some $\sigma=(\sigma_j)_{j\in J}\in \{\pm 1\}^{(J)}$. Let $I\subseteq J$ be the reunion of the supports of $x$, $w$ and $\sigma$. Pick any $i_0\in J\setminus I$, and let $\sigma^{i_0}$ denote the element of $\{\pm 1\}^{(J)}$ with support $\{i_0\}$. Then one may choose
$x'\in\{0, e_{i_0}\}\subseteq Q$ and $\sigma'\in \{\mathbf{1},\sigma^{i_0}\}\subseteq \{\pm 1\}^{(J)}$
such that
$$\widehat{w}':=\tau_{x+x'}\sigma\sigma' w\inv\in \widehat{\WW}(D^{(1)}_J).$$
Moreover, (\ref{eqn:second_eqn}) yields that
\begin{equation*}
\begin{aligned}
\la(\widehat{w}.\chi-\chi)&=\frac{1}{2}\sum_{j\in I}{\Big((n_j+\la_j+\sigma_jd_{w(j)})^2-(\la_j+d_j)^2\Big)}\\
&\geq \la(\widehat{w}'.\chi-\chi)-\frac{1}{2}\Big((1+|\la_{i_0}|+|d_{i_0}|)^2+(|\la_{i_0}|+|d_{i_0}|)^2\Big)\\
&\geq \la(\widehat{w}'.\chi-\chi)-(3/2+C)^2.
\end{aligned}
\end{equation*}
Hence $(J,\la,\chi)$ satisfies the PEC for $\widehat{\WW}(C^{(1)}_J)$, as desired.
\end{proof}

\begin{prop}\label{prop:B2J}
Let $\la=(1,\la^0,0)$ and $\chi=(0,\chi^0,1)$, and set $\la_2:=(1,\tfrac{\la^0}{2},0)$ and $\chi_2:=(0,\tfrac{\chi^0}{2},1)$. Then 
$(J,\la,\chi)$ satisfies the PEC for $\widehat{\WW}(B^{(2)}_J)$ if and only if $(J,\la_2,\chi_2)$ satisfies the PEC for $\widehat{\WW}(C^{(1)}_J)$.
\end{prop}
\begin{proof}
This readily follows from Lemma~\ref{lemma:minset_B2_aff}.
\end{proof}

Using Lemmas~\ref{lemma:SC0}, \ref{lemma:simpleobs2} and \ref{lemma:translation_invariance1}, we can now restate the results of this section for general $\la=(\la_c,\la^{0},\la_d)$ and $\chi=(\chi_c,\chi^{0},\chi_d)$ with $\la_c\chi_d\neq 0$. Recall the definition of the cone $C_{\min}(\la,X)$ in this setting 
(Definition~\ref{def:5.10}). We denote again by $\ell^1(J)$ the set of $(\chi_c,\chi^0,\chi_d)\in\RR\times\RR^J\times\RR$ with $\chi_c=\chi_d=0$ and $\chi^0\in\ell^1(J)$. 

\begin{theorem}\label{thm:PECgen}
Let $\la=(\la_c,\la^0,\la_d)\in\RR\times\RR^J\times\RR$ with $\la_c\neq 0$ be $\ZZ$-discrete, and let $X=X_J^{(1)}$ or $X_J^{(2)}$ be one of the seven standard types from \S\ref{section:preliminaries:affine}. Then for any $\chi=(\chi_c,\chi^0,\chi_d)\in\RR\times\RR^J\times\RR$ with $\chi_d\neq 0$, the following assertions are equivalent:
\begin{enumerate}
\item
$(J,\la,\chi)$ satisfies the PEC for $\widehat{\WW}(X)$.
\item
$\la_c\chi_d>0$ and $\chi\in C_{\min}(\la,X)+\ell^1(J)$.
\end{enumerate}
\end{theorem}
\begin{proof}
For $X=A_J^{(1)}$, this follows from Theorem~\ref{thm:PECA}. For $X=C^{(1)}_J$, this follows from Theorem~\ref{thm:PECB}. For $X\in\{B^{(1)}_J, D^{(1)}_J, C^{(2)}_J, BC^{(2)}_J\}$, the implication (1)$\Rightarrow$(2) follows from Theorem~\ref{thm:PECB} and Proposition~\ref{prop:other_types}, together with the fact that $C_{\min}(\la,C^{(1)}_J)\subseteq C_{\min}(\la,X)$ as $\widehat{\WW}(X)\subseteq \widehat{\WW}(C^{(1)}_J)$. Conversely, (2)$\Rightarrow$(1) follows from Lemma~\ref{lemma:invariance_summable_aff2}. Similarly, for $X=B_J^{(2)}$, the implication (1)$\Rightarrow$(2) follows from Theorem~\ref{thm:PECB} and Proposition~\ref{prop:B2J}, together with Lemma~\ref{lemma:minset_B2_aff}. Conversely, (2)$\Rightarrow$(1) follows from Lemma~\ref{lemma:invariance_summable_aff2}.
\end{proof}

\noindent
{\bf Proof of Theorem~\ref{thmintro:A}.}
Since $\la$ is $\ZZ$-discrete by Lemma~\ref{lemma:integral-Zdiscrete}, this follows from Theorem~\ref{thm:PECgen}. \hspace{\fill}\qedsymbol

\section{Positive energy representations of double extensions of Hilbert loop algebras}\label{section:PERDEHLA}
We conclude this paper with a more precise statement of Corollary~\ref{corintro:arbitrary}. As announced in the introduction, Corollary~\ref{corintro:arbitrary} can be deduced from Theorem~\ref{thmintro:A} by using the main results of \cite{PECisom}. However, as noted in Remark~\ref{remark:PECisom1}, the definition of loop algebras in this paper and in the paper \cite{PECisom} slightly differ, and we now do the extra work required to pass from one convention to the other.

Let $\kk$ be a simple Hilbert--Lie algebra and $\varphi\in\Aut(\kk)$ be an automorphism of $\kk$ of finite order $N_{\varphi}$. For $N\in\NN$ we set, as in Remark~\ref{remark:PECisom1}, $$\LLL_{\varphi,N}(\kk):=\{\xi\in C^{\infty}(\RR,\kk) \ | \ \xi(t+\tfrac{2\pi}{N})=\varphi\inv(\xi(t)) \ \forall t\in\RR\}.$$
The convention in the present paper is thus to take $N=N_{\varphi}$, while the convention in \cite{PECisom} is to take $N=1$. Let $\ttt_0$ be a maximal abelian subalgebra of $\kk^{\varphi}$, and for a weight $\nu\in i\ttt_0^*$, consider as in the introduction the double extension
$$\widehat{\LLL}^{\nu}_{\varphi,N}(\kk):=(\RR\oplus_{\omega_{D_{\nu}}}\LLL_{\varphi,N}(\kk))\rtimes_{\widetilde{D}_{\nu}}\RR$$
of $\LLL_{\varphi,N}(\kk)$, with Cartan subalgebra $\ttt_0^e:=\RR\oplus\ttt_0\oplus\RR$. We respectively denote by $\widehat{\WW}_{\varphi}^{\nu}\subseteq\GL(i(\ttt_0^e)^*)$ and $\widehat{\Delta}_{\varphi}\subseteq i(\ttt_0^e)^*$ the Weyl group and root system of 
$$\g:=\widehat{\LLL}^{\nu}_{\varphi,N_{\varphi}}(\kk)$$ with respect to $\ttt_0^e$.

Let $\bc=(i,0,0)\in i\ttt_0^e$ and $\bd=(0,0,-i)\in i\ttt_0^e$. Recall from Remark~\ref{remark:PECisom1} the identification 
$$\widehat{V}_{\fin}^{(2)}\stackrel{\sim}{\to} i\ttt_0^e:(z,h,t)\mapsto z\bc+h+t\bd,$$ 
as well as its extension $\widehat{V}\stackrel{\sim}{\to}i\widehat{\ttt_0^e}$. Since we are now working in $i\widehat{\ttt_0^e}$ instead of $\widehat{V}$, we will write to avoid any confusion $\chi=[\chi_c,\chi^0,\chi_d]$ for the element $\chi_c\bc+\chi^0+\chi_d\bd=(i\chi_c,\chi^0,-i\chi_d)$ of $i\widehat{\ttt_0^e}$ and $\la=[\la_c,\la^0,\la_d]$ for the weight $\la\in i(\ttt_0^e)^*$ with $\la(\bc)=\la_c$, $\la|_{i\ttt_0}=\la^0$ and $\la(\bd)=\la_d$. 

Let $\nu\in i\ttt_0^*$. By \cite[Theorem~A]{PECisom}, one can choose the Cartan subalgebra $\ttt_0$ such that there exists a weight $\mu\in i\ttt_0^*$ and an isomorphism $\widehat{\LLL}_{\varphi,1}^{\nu}(\kk)\stackrel{\sim}{\to}\widehat{\LLL}_{\psi,1}^{\mu+\nu}(\kk)$ from $\widehat{\LLL}_{\varphi,1}^{\nu}(\kk)$ to one of the seven (slanted) standard affinisations of $\kk$ fixing the common Cartan subalgebra $\ttt_0^e$ pointwise. To distinguish between these Cartan subalgebras, we will also write $\ttt_0^e(\varphi)=\ttt_0^e$ (resp. $\ttt_0^e(\psi)=\ttt_0^e$) when $\ttt_0^e$ is viewed as a subalgebra of $\widehat{\LLL}_{\varphi,1}^{\nu}(\kk)$ (resp. $\widehat{\LLL}_{\psi,1}^{\mu+\nu}(\kk)$). 

On the other hand, by \cite[Remark~4.3]{PECisom}, there is for each $N\in\NN$ and $\phi\in\{\varphi,\psi\}$ an isomorphism 
$\widehat{\LLL}_{\phi,1}^{\thinspace \nu/N}(\kk)\stackrel{\sim}{\to} \widehat{\LLL}_{\phi,N}^{\nu}(\kk)$
whose $\CC$-linear extension to the corresponding complexifications restricts to the isomorphism 
$$i\ttt_0^e\stackrel{\sim}{\to}i\ttt_0^e:[z,h,t]\mapsto [Nz,h,t/N].$$
Here we use the same notation for the Cartan subalgebras of $\widehat{\LLL}_{\phi,1}^{\thinspace \nu/N}(\kk)$ and $\widehat{\LLL}_{\phi,N}^{\nu}(\kk)$.

Let $N_{\psi}\in\{1,2\}$ denote the order of $\psi$, and set $Q:=N_{\psi}/N_{\varphi}$.
Composing the above isomorphisms yields an isomorphism
\begin{equation}\label{eqn:isom}
\g=\widehat{\LLL}_{\varphi,N_{\varphi}}^{\nu}(\kk)\stackrel{\sim}{\to} \widehat{\LLL}_{\psi,N_{\psi}}^{\thinspace Q\nu+N_{\psi}\mu}(\kk)
\end{equation}
whose $\CC$-linear extension to the corresponding complexifications restricts to the isomorphism 
\begin{equation}\label{eqn:Cartansub}
\Phi\co i\ttt_0^e(\varphi)\stackrel{\sim}{\to}i\ttt_0^e(\psi):[z,h,t]\mapsto [Qz,h,t/Q].
\end{equation}
Note that
\begin{equation}\label{eqn:Weylgr}
\widehat{\WW}_{\varphi}^{\nu}=\Phi\inv\widehat{\WW}_{\psi}^{\thinspace Q\nu+N_{\psi}\mu}\Phi\subseteq\GL(i\ttt_0^e(\varphi)).
\end{equation}

Let now $\la_{\varphi}=[\la_c,\la^0,\la_d]\in i\ttt_0^e(\varphi)^*$ be an integral weight for $\g$ with $\la_c\neq 0$. Let also $\nu'\in i\ttt_0^e(\varphi)$, and consider as in the introduction the corresponding highest weight representation $$\widetilde{\rho}=\widetilde{\rho}_{\la_{\varphi},\chi_{\varphi}}\co \g \rtimes \RR D_{\nu'}\to\End(L(\la_{\varphi})),$$
where $$\chi_{\varphi}\co\ZZ[\widehat{\Delta}_{\varphi}]\to\RR:(\alpha,n)\mapsto n+\nu'(\alpha^{\sharp}).$$
As in Remark~\ref{remark:lachi_chila}, we view the character $\chi_{\varphi}$ as an 
element of $i\widehat{\ttt_0^e}$: by \cite[\S 7.2 Eq.~(7.7)]{PECisom} we then get
$$\chi_{\varphi}=[0,(\nu')^{\sharp}-\nu^{\sharp},1].$$
We recall that the representation $\widetilde{\rho}$ is of positive energy if and only if the set
$$\mathbf{E}:=\la_{\varphi}\big(\widehat{\WW}_{\varphi}^{\nu}.\chi_{\varphi}-\chi_{\varphi}\big)$$
is bounded from below.
By (\ref{eqn:Weylgr}), we can rewrite this set as 
$$\mathbf{E}=\la_{\psi}\big(\widehat{\WW}_{\psi}^{\mm}.\chi_{\psi}-\chi_{\psi}\big)$$
where
$$\mm:=Q\nu+N_{\psi}\mu, \quad\la_{\psi}:=\la_{\varphi}\circ\Phi\inv=[\la_c/Q,\la^{0},Q\la_d]\quad\textrm{and}\quad \chi_{\psi}:=\Phi(\chi_{\varphi})=[0,(\nu')^{\sharp}-\nu^{\sharp},1/Q].$$
Set 
$$\la:=\big[1,Q\tfrac{\la^{0}}{\la_c},0\big], \quad\textrm{and}\quad \chi:=\big[0,Q((\nu')^{\sharp}-\nu^{\sharp}),1\big].$$
In view of Lemma~\ref{lemma:simpleobs2} and \cite[Proposition~7.4]{PECisom}, we have in turn that
\begin{equation}\label{eqn:E}
\mathbf{E}=\tfrac{\la_c}{Q^2} \la\big(\widehat{\WW}_{\psi}^{\mm}.\chi-\chi\big)=\tfrac{\la_c}{Q^2} \la_{\mm}\big(\widehat{\WW}_{\psi}^{0}.\chi_{\mm}-\chi_{\mm}\big),
\end{equation}
where $\widehat{\WW}^0_{\psi}$ is the (standard) Weyl group of $\widehat{\LLL}_{\psi,N_{\psi}}^{0}(\kk)$, hence one of the $7$ Weyl groups $\widehat{\WW}(X)$ for $X=X_J^{(1)}$ or $X=X_J^{(2)}$ described in \S\ref{subsection:Weyl_group}. Note that
$$\la_{\mm}=\big[1,Q\tfrac{\la^{0}}{\la_c}-\mm,0\big]=\big[1,Q(\tfrac{\la^{0}}{\la_c}-\nu-N_{\psi}\mu),0\big]=\tfrac{Q}{\la_c}\cdot\big[\tfrac{\la_c}{Q},\la^{0}-\la_c(\nu+N_{\varphi}\mu),0\big]$$ 
and
$$\chi_{\mm}=\big[0,Q((\nu')^{\sharp}-\nu^{\sharp})+\mm,1\big]=\big[0,Q((\nu')^{\sharp}+N_{\varphi} \mu^{\sharp}),1\big]=Q\cdot\big[0,(\nu')^{\sharp}+N_{\varphi}\mu^{\sharp},1/Q\big].$$

The following theorem summarises the above discussion.
\begin{theorem}\label{thm:CorBprecise}
Let $\g=\widehat{\LLL}^{\nu}_{\varphi,N_{\varphi}}(\kk)$ be an arbitrary affinisation of a simple Hilbert--Lie algebra $\kk$. Let $\ttt_0$ be a Cartan subalgebra of $\kk^{\varphi}$ such there is some $\mu\in i\ttt_0^*$ and some standard (or trivial) automorphism $\psi\in\Aut(\kk)$ for which \eqref{eqn:isom} and \eqref{eqn:Cartansub} hold. Let $X=X_J^{(1)}$ or $X_J^{(2)}$ be the type of the root system of the standard affinisation $\widehat{\LLL}^{0}_{\psi,N_{\psi}}(\kk)$ of $\kk$ and set $Q:=N_{\psi}/N_{\varphi}$. 
Finally, let $\la_{\varphi}=[\la_c,\la^0,\la_d]\in i(\ttt_0^e)^*$ be an integral weight for $\g$ with $\la_c\neq 0$ and let $\nu'\in i(\ttt_0^e)^*$. Set $\chi_{\varphi}:=[0,(\nu')^{\sharp}-\nu^{\sharp},1]$. Then the following assertions are equivalent:
\begin{enumerate}
\item
The highest weight representation $\widetilde{\rho}_{\la_{\varphi},\chi_{\varphi}}\co \g \rtimes \RR D_{\nu'}\to\End(L(\la_{\varphi}))$ is of positive energy.
\item
The set $\mathbf{E}:=\la_{\varphi}\big(\widehat{\WW}_{\varphi}^{\nu}.\chi_{\varphi}-\chi_{\varphi}\big)$ is bounded from below.
\item
The triple $\big(J,\big[\tfrac{\la_c}{Q},\la^{0}-\la_c(\nu+N_{\varphi}\mu),0\big],\big[0,(\nu')^{\sharp}+N_{\varphi}\mu^{\sharp},1/Q\big]\big)$ satisfies the PEC for $\widehat{\WW}(X)$.
\item
$\chi_{\varphi}=\chi_{\varphi}^{\min}+\chi_{\varphi}^{\su}$ for some minimal energy character $\chi_{\varphi}^{\min}$, satisfying  $\inf\la_{\varphi}\big(\widehat{\WW}_{\varphi}^{\nu}.\chi_{\varphi}^{\min}-\chi_{\varphi}^{\min}\big)=\{0\}$, and some summable character $\chi_{\varphi}^{\su}\in\ell^1(J)$.
\end{enumerate}
Moreover, $\inf\mathbf{E}=0$ if and only if the triple $\big(J,\big[\tfrac{\la_c}{Q},\la^{0}-\la_c(\nu+N_{\varphi}\mu),0\big],\big[0,(\nu')^{\sharp}+N_{\varphi}\mu^{\sharp},1/Q\big]\big)$ is of minimal energy for $\widehat{\WW}(X)$.
\end{theorem}
\begin{proof}
The equivalence of (1), (2) and (3), as well as the last statement of the theorem readily follow from the above discussion.
The equivalence of (3) and (4) follows from Theorem~\ref{thm:PECgen} and (\ref{eqn:E}). Note that $[\tfrac{\la_c}{Q},\la^{0}-\la_c(\nu+N_{\varphi}\mu),0]$ is indeed $\ZZ$-discrete: this can be seen as in the proof of Lemma~\ref{lemma:integral-Zdiscrete}. More precisely, set $\la^0=(\la_j)_{j\in J}$, $\nu=(\nu_j)_{j\in J}$ and $\mu=(\mu_j)_{j\in J}$. The integrality condition on $\la_{\varphi}$ implies that 
$$\la_{\varphi}((0,\epsilon_i-\epsilon_j,n)^{\vee})=\la_{\varphi}((-n-(\nu_i-\nu_j),e_i-e_j,0))=-(n+\nu_i-\nu_j)\la_c+\la_i-\la_j\in\ZZ$$
for infinitely many values of $n\in\ZZ$ (see \cite[\S 3.4]{PECisom}). Hence $\la_c$ is rational, say $\la_c=m/p$ for some nonzero integers $m,p$. Then
$$\big(\tfrac{\la_i}{\la_c}-\nu_i\big)-\big(\tfrac{\la_j}{\la_c}-\nu_j\big)\in\ZZ+\tfrac{1}{m}\ZZ$$
for all $i,j\in J$, and hence
$$\big\{\tfrac{\la_j}{\la_c}-\nu_j+\ZZ \ | \ j\in J\big\}$$
is finite. This implies in turn that
$$\big\{Q(\tfrac{\la_j}{\la_c}-\nu_j-N_{\varphi}\mu_j)+\ZZ \ | \ j\in J\big\}$$
is finite because $Q$ is rational and $\{\mu_j \ | \ j\in J\}$ is a finite subset of $\QQ$ (in fact, $\mu_j$ is of the form $\mu_j=-n_j/N$ for some $N\in\{N_{\varphi},2N_{\varphi}\}$ and some $n_j\in\{0,1,\dots,N-1\}$, see \cite[Section~6]{PECisom}). This yields the claim. 
\end{proof}

\bibliographystyle{amsalpha} 
\bibliography{these}

\def\cprime{$'$}
\providecommand{\bysame}{\leavevmode\hbox to3em{\hrulefill}\thinspace}
\providecommand{\MR}{\relax\ifhmode\unskip\space\fi MR }
\providecommand{\MRhref}[2]{%
  \href{http://www.ams.org/mathscinet-getitem?mr=#1}{#2}
}
\providecommand{\href}[2]{#2}
\begin{thebibliography}{MN15b}

\bibitem[BR87]{BR02}
Ola Bratteli and Derek~W. Robinson, \emph{Operator algebras and quantum
  statistical mechanics. 1}, second ed., Texts and Monographs in Physics,
  Springer-Verlag, New York, 1987, $C{^{\ast}}$- and $W{^{\ast}}$-algebras,
  symmetry groups, decomposition of states.

\bibitem[HN12]{convexhull}
Georg Hofmann and Karl-Hermann Neeb, \emph{On convex hulls of orbits of
  {C}oxeter groups and {W}eyl groups}, Preprint (2012), to appear in Münster
  Journal of Mathematics, http://arxiv.org/abs/1204.2095.

\bibitem[LN04]{LN04}
Ottmar Loos and Erhard Neher, \emph{Locally finite root systems}, Mem. Amer.
  Math. Soc. \textbf{171} (2004), no.~811, x+214.

\bibitem[MN15a]{PECisom}
Timoth\'ee Marquis and Karl-Hermann Neeb, \emph{Isomorphisms of twisted
  {H}ilbert loop algebras}, Preprint (2015), http://arxiv.org/abs/1508.07938v2.

\bibitem[MN15b]{PEClocfin}
\bysame, \emph{Positive energy representations for locally finite split {L}ie
  algebras}, Preprint (2015), to appear in Int. Math. Res. Not. IMRN,
  http://arxiv.org/abs/1507.06077.

\bibitem[MY15]{MY15}
Jun Morita and Yoji Yoshii, \emph{Locally loop algebras and locally affine
  {L}ie algebras}, J. Algebra \textbf{440} (2015), 379--442.

\bibitem[Nee10]{Ne09}
Karl-Hermann Neeb, \emph{Unitary highest weight modules of locally affine {L}ie
  algebras}, Quantum affine algebras, extended affine {L}ie algebras, and their
  applications, Contemp. Math., vol. 506, Amer. Math. Soc., Providence, RI,
  2010, pp.~227--262.

\bibitem[Nee14a]{Ne14}
\bysame, \emph{Positive energy representations and continuity of projective
  representations for general topological groups}, Glasg. Math. J. \textbf{56}
  (2014), no.~2, 295--316.

\bibitem[Nee14b]{Hloopgroups}
\bysame, \emph{Semibounded unitary representations of double extensions of
  {H}ilbert-loop groups}, Ann. Inst. Fourier (Grenoble) \textbf{64} (2014),
  no.~5, 1823--1892.

\bibitem[NS01]{NeSt01}
Karl-Hermann Neeb and Nina Stumme, \emph{The classification of locally finite
  split simple {L}ie algebras}, J. Reine Angew. Math. \textbf{533} (2001),
  25--53.

\bibitem[Yos10]{YY08}
Yoji Yoshii, \emph{Locally extended affine root systems}, Quantum affine
  algebras, extended affine {L}ie algebras, and their applications, Contemp.
  Math., vol. 506, Amer. Math. Soc., Providence, RI, 2010, pp.~285--302.

\end{thebibliography}

\end{document}